\documentclass[11pt, a4paper, reqno]{amsart}
\usepackage{amsmath}
\usepackage{amsfonts}
\usepackage{amssymb}
\usepackage{amsthm}
\usepackage{mathtools} 
\usepackage{tikz} 
\usepackage{url}
\usepackage{color}
\usepackage{array}
\usepackage[utf8x]{inputenc}
\usepackage{t1enc}
\usepackage[english]{babel}

\usepackage{amsmath,amssymb,amsthm}
\usepackage{mathrsfs,esint,bm} 
\usepackage{graphicx,wrapfig}
\usepackage[format=hang]{caption}
\usepackage[shortlabels]{enumitem}

\newcommand*{\R}{{\mathbb R}}

\newcommand*{\N}{{\mathbb N}}

\newcommand*{\eps}{\varepsilon}

\newcommand*{\calE}{\mathcal{E}}

\providecommand*{\vint}[1]{\mathchoice
          {\mathop{\vrule width 5pt height 3 pt depth -2.5pt
                  \kern -9pt \kern 1pt\intop}\nolimits_{\kern -5pt{#1}}}
          {\mathop{\vrule width 5pt height 3 pt depth -2.6pt
                  \kern -6pt \intop}\nolimits_{\kern -3pt{#1}}}
          {\mathop{\vrule width 5pt height 3 pt depth -2.6pt
                  \kern -6pt \intop}\nolimits_{\kern -3pt{#1}}}
          {\mathop{\vrule width 5pt height 3 pt depth -2.6pt
                  \kern -6pt \intop}\nolimits_{\kern -3pt{#1}}}}

\DeclareMathOperator{\diam}{diam}

\DeclareMathOperator*{\essinf}{ess\,inf} 
\newcommand{\norm}[1]{\left\lVert#1\right\rVert}  
\newcommand{\abs}[1]{\left\lvert#1\right\rvert}  
\newcommand{\hdim}{d_{\mathrm{f}}}  
\newcommand{\pwalk}{d_{\mathrm{w},p}}  

\numberwithin{equation}{section}
\theoremstyle{plain}
\newtheorem{thm}[equation]{Theorem}

\newtheorem{prop}[equation]{Proposition}
\newtheorem{cor}[equation]{Corollary}
\newtheorem{lem}[equation]{Lemma}

\theoremstyle{definition}

\newtheorem{defn}[equation]{Definition}

\newtheorem{remark}[equation]{Remark}

\newtheorem{example}[equation]{Example}

\begin{document}

\title[Finite dimensionality of Besov spaces, and decomposition]
{Finite dimensionality of Besov spaces and potential-theoretic decomposition of metric spaces}
\author[T. Kumagai]{Takashi Kumagai}
\address{T.K.: Department of Mathematics,
Waseda University, 3-4-1 Okubo, Shinjuku-ku, Tokyo 169-8555, Japan.}
\email{t-kumagai@waseda.jp}
\author[N. Shanmugalingam]{Nageswari Shanmugalingam}
\address{N.S.: Department of Mathematical Sciences, P.O.~Box 210025, University of Cincinnati, Cincinnati, OH~45221-0025, U.S.A.}
\email{shanmun@uc.edu}
\author[R. Shimizu]{Ryosuke Shimizu}
\address{R.S.: Waseda Research Institute for Science and Engineering,
Waseda University, 3-4-1 Okubo, Shinjuku-ku, Tokyo 169-8555, Japan.}
\curraddr{G{\tiny RADUATE} S{\tiny CHOOL} {\tiny OF} I{\tiny NFORMATICS}, K{\tiny YOTO} U{\tiny NIVERSITY}, Y{\tiny OSHIDA-HONMACHI}, S{\tiny AKYO-KU}, K{\tiny YOTO} 606-8501, J{\tiny APAN}.}
\email{r.shimizu@acs.i.kyoto-u.ac.jp}
\thanks{
 T.K.'s work is partially supported by JSPS KAKENHI Grant Numbers 22H00099 and 23KK0050.
N.S.'s work is partially supported by the NSF (U.S.A.) grant DMS~\#2054960.
R.S.'s work (JSPS Research Fellow-PD) is partially supported by
JSPS KAKENHI Grant Number 23KJ2011. The authors thank the referee for the kind suggestions that helped improve the exposition of the paper.}
\maketitle

\begin{abstract}
In the context of a metric measure space $(X,d,\mu)$, we explore the
potential-theoretic implications of having a finite-dimensional Besov space.
We prove that if the dimension of the Besov space $B^\theta_{p,p}(X)$ is $k>1$,
then $X$ can be decomposed into $k$ number of irreducible components (Theorem \ref{thm:main1}).
Note that $\theta$ may be bigger than $1$, as our framework includes fractals.
We also provide sufficient conditions under which the dimension of the Besov space is $1$.
We introduce critical exponents $\theta_p(X)$ and $\theta_p^{\ast}(X)$ for the Besov spaces.
As examples illustrating Theorem~\ref{thm:main1}, we compute these critical exponents for spaces $X$
formed by glueing copies of $n$-dimensional cubes, the Sierpi\'{n}ski gaskets, and of the Sierpi\'{n}ski carpet.
\end{abstract}

\noindent
    {\small \emph{Key words and phrases}:
Besov spaces, Korevaar-Schoen spaces, fractal, irreducible $p$-energy form,
Newton-Sobolev spaces, $p$-Poincar\'e inequality, Sierpi\'nski fractals, decomposition}.

\medskip

\noindent
    {\small Mathematics Subject Classification (2020):
Primary: 31E05, 28A80;
Secondary: 46E36, 31C25}

\section{Introduction}\label{Sec:1}

Given a compact metric space $(X,d)$ equipped with a doubling measure $\mu$, a viable theory of local Dirichlet-type
energy forms is obtained by considering the Newton-Sobolev class $N^{1,p}(X)$
of functions on $X$ if we know that $(X,d,\mu)$ supports a $p$-Poincar\'e inequality for
some $1\le p<\infty$. However, when no Poincar\'e type inequality is available on $(X,d,\mu)$, a more natural local energy
form is given by the so-called Korevaar-Schoen space $KS^1_p(X)$,
see for instance~\cite{KM}.
We are interested in the  
function-classes
$B^\theta_{p,p}(X)$ (Besov), $B^\theta_{p,\infty}(X)$,
and $KS^\theta_p(X)$
(Korevaar-Schoen).
These are spaces of functions in $L^p(X)$ for which the following respective energies are finite:
\begin{align*}
||u||_{B^\theta_{p,p}(X)}^p&:=\int_X\int_X\frac{|u(y)-u(x)|^p}{d(x,y)^{\theta p}\, \mu(B(x,d(x,y)))}\, d\mu(y)\, d\mu(x)\\
  &\qquad\qquad\approx \int_0^{\diam(X)} \int_X\vint{B(x,t)}\frac{|u(y)-u(x)|^p}{t^{\theta p}}\, d\mu(y)\, d\mu(x)\, \frac{dt}{t};\\
||u||_{B^\theta_{p,\infty}(X)}^p&:=\sup_{t>0}\, \int_X\vint{B(x,t)}\frac{|u(y)-u(x)|^p}{t^{\theta p}}\, d\mu(y)\, d\mu(x);\\
||u||_{KS^\theta_p(X)}^p&:=\limsup_{t\to 0^+}\,\int_X\vint{B(x,t)}\frac{|u(y)-u(x)|^p}{t^{\theta p}}\, d\mu(y)\, d\mu(x),
\end{align*}
where, by $F\approx H$ we mean that there is a constant $C\ge 1$, independent of the parameters
$F$ and $H$ depend on (in the above it would be $u$), so that $C^{-1}\le F/H\le C$.
(For the equivalence on $||u||_{B^\theta_{p,p}(X)}^p$ under the volume doubling property, see \cite[Theorem 5.2]{GKS}.)
While the energy $||u||_{KS^\theta_p(X)}$ is local, the energy $||u||_{B^\theta_{p,\infty}(X)}$ is not.
In general we do not know that the two norms $||u||_{B^\theta_{p,\infty}(X)}$ and $||u||_{KS^\theta_p(X)}$ are comparable,
but because $\mu$ is doubling, we have that as sets, $B^\theta_{p,\infty}(X)=KS^\theta_p(X)$, see Lemma~\ref{lem:B=KS} below.

The goal of this paper is to investigate what the potential-theoretic implications are of knowing that $B^\theta_{p,p}(X)$ has finite dimension.
The following two critical exponents $\theta_{p}(X)$ and $\theta_{p}^{\ast}(X)$ for the Besov space will play important roles.
Throughout the paper, we assume that $X$ has infinitely many points.
Inspired by the ground-breaking result of
Bourgain, Brezis and Mironescu~\cite{BBM},
we define
\begin{align*}
	\theta_p(X) &\coloneqq \theta_{p} \coloneqq \sup\{\theta>0: B^\theta_{p,p}(X)~\mbox{contains non-constant functions}\}; \\
	\theta_p^{\ast}(X) &\coloneqq \theta_{p}^{\ast} \coloneqq \sup\{\theta>0: B^\theta_{p,p}(X)~\mbox{is dense in $L^{p}(X)$}\}.
\end{align*}
Note that $\theta_p(X)\ge 1$ if $(X,d,\mu)$ is a doubling metric measure space (see
Lemma~\ref{lem:theta_pge1}), and that $\theta_{p}(X) \ge \theta_{p}^{\ast}(X)$.
When the measure on $X$ is doubling and supports a $p$-Poincar\'e inequality for all function-upper gradient pairs
as in~\eqref{e:PIug}, then we must have $\theta_p=1$.
If the dimension of $B^\theta_{p,p}(X)$ is $1$, then $B^\theta_{p,p}(X)$ consists solely of constant functions and $\theta_p(X)\le \theta$.
The following theorem tells us that if the dimension of $B^\theta_{p,p}(X)$ is finite but larger than $1$, then $X$ can be decomposed
into as many pieces as the dimension of $B^\theta_{p,p}(X)$ so that there is no potential-theoretic communication between
different pieces.

\begin{thm}\label{thm:main1}
Let $(X,d,\mu)$ be a doubling metric measure space as in~\eqref{eq:defn.VD} %(see \eqref{eq:defn.VD})}  
and $\theta>0$. Suppose that the dimension of
$B^\theta_{p,p}(X)$ is finite. Then 
either $\mu(X)=\infty$ and $B^\theta_{p,p}(X)=\{0\}$ {\rm (}in which case
$\theta\ge \theta_p(X)${\rm )}, or
there exist measurable sets $E_1,\cdots, E_k$, with $k$ the dimension of $B^\theta_{p,p}(X)$,
such that the following hold:
\begin{enumerate}[\rm(1)]
\item\label{it:positivemass} $0<\mu(E_i)<\infty$ for $i=1,\cdots, k$,
\item\label{it:cover} If $\mu(X)<\infty$, then $\mu(X\setminus \bigcup_{i=1}^kE_i)=0$,
\item\label{it:basis} $\chi_{E_i}\in B^\theta_{p,p}(X)$ for $i=1,\cdots, k$, and
$\{\chi_{E_i}\, :\, i=1,\cdots, k\}$ forms a basis for $B^\theta_{p,p}(X)$.
\item\label{it:directsum} $B^\theta_{p,p}(X)=\oplus_{i=1}^k B^\theta_{p,p}(E_i):=\{f\in L^p(X): f|_{E_i}\in B^\theta_{p,p}(E_i), i=1,\cdots, k\}$ as sets.
Moreover, the dimension of $B^\theta_{p,p}(E_i)$ is $1$ for all
$i=1,\cdots, k$.
\item\label{it:zeroenergy} $||\chi_{E_i}||_{KS^\theta_p(X)}=0$ for $i=1,\cdots, k$.
\item\label{it:energydecomp} If $u\in KS^\theta_p(X)\cap L^\infty(X)$, then for $j=1,\cdots, k$ we have
\[
\|u\, \chi_{E_j}\|_{KS^\theta_p(X)}^p
%=\limsup_{r\to 0^+}\int_{E_j}\vint{B(x,r)}\frac{|u(y)-u(x)|^p}{r^{\theta p}}\, d\mu(y)\, d\mu(x).
= \limsup_{r\to 0^+}\int_{E_j}\int_{B(x,r)\cap E_j}\frac{|u(y)-u(x)|^p}{r^{\theta p}\, \mu(B(x,r))}\, d\mu(y)\, d\mu(x).
\]
\item\label{it:thetap} $\theta\le\theta_p(X)$ if $k>1$ or $\mu(X)=\infty$ with $k=1$, and $\theta\ge\theta_p(X)$ if $\mu(X)<\infty$ and $k=1$. 
\end{enumerate}
\end{thm}

In Condition~\ref{it:energydecomp} above, we do not know whether we can remove the requirement that
$u\in L^\infty(X)$.

As a consequence of the above theorem, if $k>1$, we have a decomposition of $X$ into measurable pieces $E_i$, $i=1,\cdots, k$
(up to a null-measure set) so that there is no potential theoretic communication between different pieces; this is encoded
in the claim $||\chi_{E_i}||_{KS^\theta_p(X)}=0$.
Moreover, Condition~\ref{it:directsum}
also encodes the property that $\mu(E_i\cap E_j)=0$ when $i, j\in\{1,\cdots,k\}$ with $i\ne j$.

We now introduce the notion of \emph{irreducible $p$-energy form} for convenience.

\begin{defn}[Irreducible $p$-energy form]\label{defn:irreducible}
	Assume that $\mu(X) < \infty$.
	Let $\mathcal{F}_{p}$ be a linear subspace of $L^{p}(X,\mu)$ and let $\mathcal{E}_{p} \colon \mathcal{F}_{p} \to [0,\infty)$ be such that $\mathcal{E}_{p}(\,\cdot\,)^{1/p}$ is a seminorm on $\mathcal{F}_{p}$.
	We say that $(\mathcal{E}_{p},\mathcal{F}_{p})$ is a \emph{irreducible $p$-energy form} on $(X,\mu)$ if
	whenever $u\in \mathcal{F}_p$ with $\mathcal{E}_p(u)=0$, we must have
	that $u$ is a constant function ($\mu$-a.e.). Otherwise, we say $(\mathcal{E}_{p},\mathcal{F}_{p})$ is a \emph{reducible $p$-energy form}.
\end{defn}
\begin{remark}
	The above definition is inspired by the theory of symmetric Dirichlet forms
	(i.e. $p=2$ case).
	See \cite[Theorem 2.1.11]{CF} for other (equivalent) formulations of the irreducibility of recurrent symmetric Dirichlet forms.
\end{remark}
By Theorem \ref{thm:main1}~\ref{it:zeroenergy}, we have the following;
if the dimension of $B^\theta_{p,p}(X)$ is finite and larger than $1$, then
$(\|\cdot \|_{KS^\theta_p(X)}, KS^\theta_p(X))$ is reducible.
Note that if the dimension of $B^\theta_{p,p}(X)$ is $1$ and 
$\mu(X)<\infty$, then clearly
$(\norm{\,\cdot\,}_{B^{\theta}_{p,p}(X)}^{p}$, $B^{\theta}_{p,p}(X))$ is
irreducible, and only constant functions are in $B^\theta_{p,p}(X)$.
Next we provide a sufficient condition regarding the behaviors of $\norm{\,\cdot\,}_{B^{\theta}_{p,p}(X)}$ and
of $\norm{\,\cdot\,}_{KS^{\theta}_{p}(X)}$ under which the dimension of $B^{\theta}_{p,p}(X)$ is $1$.

\begin{defn}
We say that $X$ satisfies the \emph{weak maximality property}, or  \ref{e:wmax} property,
for $B^\theta_{p,\infty}(X)$ if there is a constant $C\ge 1$ such that for each
$u\in B^\theta_{p,\infty}(X)$ we have that
\begin{equation}\label{e:wmax}
||u||_{B^\theta_{p,\infty}(X)}\le C\, ||u||_{KS^\theta_p(X)}. \tag*{\textup{(w-max)$_{p,\theta}$}}
\end{equation}
\end{defn}

\begin{thm}\label{thm:main2}
We fix $1<p<\infty$ and $\theta>0$. If $(X,d,\mu)$ is a doubling metric measure space that satisfies the \ref{e:wmax} property
for $B^\theta_{p,\infty}(X)$,
then the dimension of $B^{\theta}_{p,p}(X)$ is at most $1$, and $\theta_p(X)\le \theta$. 
\end{thm}

In the spirit of \cite{Bre} we prove the following theorem, which also gives a sufficient condition for the
dimension of $B^{\theta}_{p,p}(X)$ to be at most $1$.
For $p = 2$, a similar result was proved in \cite{PP} under certain estimates on the heat kernel, in particular, the cases of
Sierpi\'{n}ski gasket and the Sierpi\'{n}ski carpet are included in \cite{PP}.

\begin{thm}\label{thm:main3.general}
	Let $1<p<\infty$ and $(X,d,\mu)$ be a doubling metric measure space.
	Assume that $(X,d,\mu)$ supports the following Sobolev-type inequality: there exist positive real numbers
	$\theta, C$ such that for
	any $u \in B^{\theta}_{p,p}(X)$, 
	\begin{align}\label{e:KSSob}
		\int_{X}\abs{u - u_{X}}^{p}\,d\mu \le C\liminf_{t \to 0^+}\int_{X}\vint{B(x,t)}\frac{\abs{u(x) - u(y)}^{p}}{t^{\theta p}}\,d\mu(y)\,d\mu(x).
	\end{align}
	Then for that choice of $\theta$ we have that $B^{\theta}_{p,p}(X)$ has %at most 
	dimension at most $1$.
\end{thm}

In the case that $(X,d,\mu)$ supports a $p$-Poincar\'{e} inequality for function--upper gradient pairs, it is known that $N^{1,p}(X) = KS^{1}_{p}(X)$ (see, e.g., ~\cite[Section~4]{KM}
or~\cite[Section~10.4, Theorem~10.4.3, and Corollary~10.4.6]{HKSTbook}) and that $\theta_{p}(X) = 1$ (see~\cite[Theorem 5.1]{Bau}).
These facts, along with Theorem~\ref{thm:main3.general}, imply the following corollary.

\begin{cor}\label{cor:main3}
	Suppose that $1<p<\infty$ and $(X,d,\mu)$ is a doubling metric measure space that supports a $p$-Poincar\'e inequality for function--upper gradient pairs (see \eqref{e:PIug}).
	Then $\theta_p(X)=1$ and $B^{1}_{p,p}(X)$ has at most dimension $1$.
\end{cor}

We emphasize that, in Theorems~\ref{thm:main1}, \ref{thm:main2}, and \ref{thm:main3.general}, we do not confine ourselves to
the case $0 < \theta \le 1$ in view of some recent studies of Sobolev spaces on fractals; see, e.g., \cite{Bau,KS2,Kig,MS,Shi}.
For example, in the case that $X$ is the Sierpi\'{n}ski carpet, M.~Murugan and the third-named author \cite{MS} proposed a
way to define the $(1,p)$-Sobolev space $\mathcal{F}_{p}$ on $X$ through discrete approximations of $X$, and 
demonstrated 
that $\mathcal{F}_{p} = KS^{\pwalk/p}_{p}(X)$ (see \cite[Theorem 7.1]{MS}) with $\pwalk > p$ (see \cite[Theorem 2.27]{Shi}).
Hence a Korevaar--Schoen space $KS^{\theta}_{p}(X)$ with $\theta > 1$ appears as a function space
playing the role of a $(1,p)$-Sobolev space on a fractal space.
Here the parameter $\pwalk$ is called the $p$-walk dimension of the Sierpi\'nski carpet, $X$, given by
$\pwalk \coloneqq \log{(8\rho_{p})}/\log{3}$, where $\rho_{p} \in (0,\infty)$ is a value called the
$p$-scaling factor of $X$ as defined in \cite[Theorem 8.5 and Definition 8.7]{MS}, 
$3$ is the reciprocal of the common contraction ratio of the family of similitudes
associated with $X$ and $8$ is the number of similitudes in this family.
(For $X = [0,1]^{n}$, we can decompose $X$ into $3^n$ cubes with side lengths $1/3$ and then see that the
$p$-scaling factor with respect to this decomposition is given by $3^{p - n}$. Hence $\pwalk = \log(3^{n} \cdot 3^{p - n})/\log{3} = p$.)
In the case $p = 2$, $(\rho_{2})^{-1}$
coincides with the \emph{resistance scaling factor} of $X$.
As a connection with quasiconformal geometry, it is
known that $\rho_{p} > 1$ if and only if $p > d_{\mathrm{ARC}}$, where $d_{\mathrm{ARC}}$ is the Ahlfors
regular conformal dimension of the Sierpi\'{n}ski carpet.
See \cite[Definitions 1.6]{MS} and \cite{CP} for further details on
$d_{\mathrm{ARC}}$.

When $\mu$ is doubling and $0<\theta<1$, the corresponding space $B^\theta_{p,p}(X)$ can be seen
as the trace space of a strongly local energy form on a larger space $(\Omega, \nu)$ with $X=\partial\Omega$ and $\mu$ and $\nu$ are
related in a co-dimensional manner, as demonstrated in~\cite{BBS}.
From the viewpoint of trace theorems on fractals, a Besov space $B^{\theta}_{p,p}(X)$ with $\theta \ge 1$ can appear as
indicated in \cite[Theorem 2.5 and 2.6]{HinoKum} for the case $p = 2$. 

In some circumstances
the reason for $\theta_p(X)>1$ may be due to $X$ being obtained as a gluing of smaller metric measure spaces along sets that are
too small to allow communication between these component spaces via the gluing set, as seen in Example~\ref{ex:bow-tie}
below, where the gluing set of two $n$-dimensional hypercubes is discussed. 
In this case, when $1<p<n$, we have that $\theta_p(X)=n/p>1$, but once we have decomposed $X$ into the two constituent
component cubes $E$ and $X\setminus E$, we have that $\theta_p(E)=\theta_p(X\setminus E)=1$, and
$B^\theta_{p,p}(X)$ is well-understood when $0<\theta<1$ as trace of a larger local process, and when $1\le \theta<\theta_p(X)$
as piecewise constant functions. Our main theorem,
Theorem~\ref{thm:main1}, gives a way of identifying this possibility. However, there are many situations where the need for
$\theta\ge 1$ is more integral to the space, as is the case of the Sierpi\'nski gasket and the Sierpi\'nski carpet, as explained
in the previous paragraph. For these spaces, typically, $B^\theta_{p,p}(X)$ has either
infinite dimension or dimension $1$. 

We conclude the introduction by reviewing some concrete examples discussed in this paper.
In Example~\ref{ex:bow-tie}, for $n \in \mathbb{N}$ with $n \ge 2$,
as mentioned above
we consider the metric measure space $X$ obtained as the union of two
$n$-dimensional hypercubes glued at a vertex, and observe that the dimension of $B^{1}_{p,p}(X)$ is $2$ when $1 < p < n$.
Note that each cubical component of $X$ supports a $p$-Poincar\'{e} inequality for any $p \ge 1$, while $X$ does not support a
$p$-Poincar\'{e} inequality when $1 < p \le n$.
Similar observations will be made in the case $X$ is the union of two copies of the Sierpi\'{n}ski carpet glued at a
vertex in Example \ref{ex:glueSC}; indeed, the dimension of $B^{\pwalk/p}_{p,p}(X)$ is $2$ when
$1 < p < d_{\mathrm{ARC}}$. 
Note that the Ahlfors regular conformal
dimension $d_{\mathrm{ARC}}$ and the $p$-walk dimension of the $n$-dimensional hypercube are $n$ and $p$ respectively. 
In both examples mentioned above, the two critical exponents $\theta_{p}(X)$ and $\theta_{p}^{\ast}(X)$ turn out to be different
when $1 < p < d_{\mathrm{ARC}}$.
Namely, the following holds, where $\hdim$ is the Hausdorff dimension of $X$.

\begin{thm}\label{thm:critical.gluedSC}
   	Let $X$ be one of the glued metric measure spaces as in Examples~\ref{ex:bow-tie} and %or 
	\ref{ex:glueSC}. Then
    	$\theta_{p}(X) = \frac{1}{p}\max\{ \hdim, \pwalk \}$ and $\theta_{p}^{\ast}(X) = \frac{\pwalk}{p}$.
\end{thm}

By \cite[Corollary 3.7]{BK} and \cite[Corollary 1.4]{CP}, we know that $\pwalk > \hdim$ if and only if $p > d_{\mathrm{ARC}}$, that $\pwalk < \hdim$ if and only if $p < d_{\mathrm{ARC}}$, and that $\pwalk = \hdim$ for $p = d_{\mathrm{ARC}}$ for these examples.
This result suggests that the case $1 < p < d_{\mathrm{ARC}}$ requires a careful treatment of the ``potential-theoretic decomposability'' of the underlying example spaces. 
See also \cite{CC} for a few examples of self-similar sets that have a similar spirit, and \cite{BBC} for the validity/invalidity of Poincar\'{e} type inequalities on a general \emph{bow-tie}, which is obtained by gluing two metric spaces at a point. 
Note that $d_{\mathrm{ARC}} = 1$ for the standard Sierpi\'nski gaskets (see, e.g., \cite[Theorem B.8]{KS1}), so the case $1 < p < d_{\mathrm{ARC}}$ does not occur in this example.  
If $X$ is the space obtained by gluing two copies of the Sierpi\'nski gasket, then $\theta_{p}(X) = \theta_{p}^{\ast}(X)$ holds for any $p \in (1,\infty)$; see Example \ref{ex:glueSG} and Theorem \ref{thm:critical.gluedSG}.

\section{Background and general results}

\subsection{Background}
Throughout this paper,
the triple $(X,d,\mu)$ is a separable metric space $(X,d)$, equipped with a Borel measure $\mu$; we require in this paper that $X$ has infinitely many points and that  $0<\mu(B(x,r)) < \infty$ for each $x\in X$ and $r>0$,  
where $B(x,r)$ denotes the set of all points $y \in X$ such that $d(x,y) < r$.
We also fix $p \in (1,\infty)$.  
Note that $\mu$ is $\sigma$-finite in this setting. 

We say that $(X,d,\mu)$ is a \emph{doubling metric measure space}, or $\mu$ is a \emph{doubling measure} on $(X,d)$, if there exists a constant $C_{\mathrm{D}}$ such that
\begin{equation}\label{eq:defn.VD}
0<\mu(B(x,2r)) \le C_{\mathrm{D}}\,\mu(B(x,r))<\infty \quad \text{for all $x \in X$, $r > 0$.}
\end{equation}
Without loss of generality, we may assume that $C_{\mathrm{D}} > 1$ if needed.

In this paper the primary function-spaces of interest are the Besov spaces and the Korevaar-Schoen spaces
	$B^\theta_{p,p}(X)$, $B^\theta_{p,\infty}(X)$, and $KS^\theta_p(X)$, as described at the beginning of Section~\ref{Sec:1} above.
	In addition, the Newton-Sobolev class $N^{1,p}(X)$ will play an auxiliary role, and we describe this class next.

A function $f \colon X \to [-\infty,\infty]$ is said to have a Borel function $g \colon X \to [0,\infty]$ as an \emph{upper gradient} if we have
\[
\abs{f(\gamma(a)) - f(\gamma(b))} \le \int_{\gamma}g\,ds
\]
whenever $\gamma \colon [a,b] \to X$ is a rectifiable curve with $a < b$.
(We interpret the inequality as also meaning that $\int_{\gamma}g\,ds = \infty$ whenever at least one of $f(\gamma(a)), f(\gamma(b))$ is not finite.)
We say that $f \in \widetilde{N^{1,p}}(X)$ if
\[
\norm{f}_{N^{1,p}(X)} \coloneqq  \left(\int_{X}\abs{f}^{p}\,d\mu\right)^{1/p} + \inf_{g}\left(\int_{X}g^{p}\,d\mu\right)^{1/p}
\]
is finite, where the infimum is over all upper gradients $g$ of $f$.
Then one can see that $\widetilde{N^{1,p}}(X)$ is a vector space.
For $f_{1}, f_{2} \in \widetilde{N^{1,p}}(X)$, we say that $f_{1} \sim f_{2}$ if $\norm{f_{1} - f_{2}}_{N^{1,p}(X)} = 0$.
Now the \emph{Newton--Sobolev class} $N^{1,p}(X)$ is defined as the collection of the equivalence classes with
respect to $\sim$, i.e., $N^{1,p}(X) \coloneqq \widetilde{N^{1,p}}(X)/\sim$. For more on this space
	we refer the interested reader to~\cite{HKSTbook}. 

We say that $(X,d,\mu)$ supports a \emph{$p$-Poincar\'{e} inequality} (with
respect to upper gradients) if there are constants $C > 0$ and $\lambda \ge 1$ such that for every
measurable function $f$ on $X$ and every upper gradient $g$ of $f$ and ball $B(x,r)$,
\begin{equation}\label{e:PIug}
	\vint{B(x,r)}\abs{f - f_{B(x,r)}}\,d\mu \le Cr\left(\vint{B(x,\lambda r)}g^{p}\,d\mu\right)^{1/p}.
\end{equation}

 From~\cite[Theorem~4.1]{KM} or~\cite[Section~10.4]{HKSTbook} we know that
if $u\in L^p(X)$ such that there is a non-negative function $g\in L^p(X)$ with $(u,g)$ satisfying the $p$-Poincar\'e inequality~\eqref{e:PIug},
then $u\in KS^1_p(X)$. In~\cite{KM} the space $KS^1_p(X)$ is denoted
by $\mathcal{L}^{1,p}(X)$.
Moreover, from~\cite[Theorems~10.5.1 and 10.5.2]{HKSTbook} we know that $KS^1_{p}(X)\subset N^{1,p}(X)$ even if
$X$ does not support a $p$-Poincar\'e inequality, and that when
$X$ supports a $p$-Poincar\'e inequality in addition, we also have $KS^1_p(X)=N^{1,p}(X)$.
Thus the index $\theta=1$ plays a key role in the theory of Sobolev spaces
in nonsmooth analysis.

\subsection{General results}
We present some lemmata on Besov spaces $B^{\theta}_{p,p}(X)$, $B^{\theta}_{p,\infty}(X)$ and the
Korevaar--Schoen space $KS^{\theta}_{p}(X)$.

\begin{lem}\label{lem:theta_pge1}
Suppose that $\mu$ is a doubling measure. Then $\theta_p(X)\ge 1$.  
\end{lem}

\begin{proof}
Fix positive $\theta<1$ and $x_0\in X$. We fix a positive number $R_0<\tfrac12\diam(X)$ so that $B(x_0,R_0)$ has
at least two points, and set
$u:X\to\R$ by
\[
u(x)=\max\{1-d(x_0,x)/R_0, 0\}.
\]
Note that $u$ is $1/R_0$-Lipschitz continuous on $X$, $0\le u\le 1$ on $X$,
and is zero outside the bounded set that is $B \coloneqq B(x_0,R_0)$.
Now
\begin{align*}
||u||_{B^\theta_{p,p}(X)}^p
&=\int_X\int_X\frac{|u(x)-u(y)|^p}{d(x,y)^{\theta p}\, \mu(B(x,d(x,y)))}\, d\mu(y)\, d\mu(x)\\
&\le\int_{2B}\int_{2B}\frac{d(x,y)^p}{R_0^p\, d(x,y)^{\theta p}\, \mu(B(x,d(x,y)))}\, d\mu(y)\, d\mu(x)\\
&\qquad\qquad\qquad+2\int_{B}\int_{X\setminus 2B}\frac{1}{d(x,y)^{\theta p}\, \mu(B(x,d(x,y)))}\, d\mu(y)\, d\mu(x).  
\end{align*}
For each non-negative integer $j$ and $x \in X$, we set $A_j(x) \coloneqq B(x,2^{j+1}R_0)\setminus B(x,2^jR_0)$. 
Since $X \setminus 2B \subset X \setminus B(x,R_0)$ for $x \in B$, we see that
\begin{align*}
\int_B\int_{X\setminus 2B}&\frac{1}{d(x,y)^{\theta p}\, \mu(B(x,d(x,y)))}\, d\mu(y)\, d\mu(x)\\
&\qquad\qquad\ \ \ \le\int_B\sum_{j=0}^\infty\int_{A_j(x)}\frac{1}{d(x,y)^{\theta p}\, \mu(B(x,d(x,y)))}\, d\mu(y)\, d\mu(x)\\
&\qquad\qquad\ \ \ \le \int_B\sum_{j=0}^\infty\int_{A_j(x)}\frac{1}{(2^jR_0)^{\theta p}\, \mu(B(x,2^jR_0))}\, d\mu(y)\, d\mu(x)\\
&\qquad\qquad\ \ \ \le \frac{\mu(B)}{R_0^{\theta p}}\, \sum_{j=0}^\infty 2^{-j\theta p}\, \frac{\mu(B(x,2^{j+1}R_0))}{\mu(B(x,2^jR_0))}\\
&\qquad\qquad\ \ \ \le \frac{2^{-\theta p}\,C_{\mathrm{D}}}{1 - 2^{-\theta p}}\frac{\mu(B)}{R_0^{\theta p}}<\infty.
\end{align*}
Moreover, setting $E_k(x):=B(x,2^{-k+2}R_0)\setminus B(x,2^{-k+1}R_0)$ for non-negative integers $k$ and $x \in X$, we have
\begin{align*}
\int_{2B}\int_{2B}&\frac{d(x,y)^p}{R_0^p\, d(x,y)^{\theta p}\, \mu(B(x,d(x,y)))}\, d\mu(y)\, d\mu(x)\\
&\quad\quad\ \ \ \le R_0^{-p}\, \int_{2B}\int_{B(x,4R_0)}\frac{d(x,y)^{(1-\theta)p}}{\mu(B(x,d(x,y)))}\, d\mu(y)\, d\mu(x)\\
&\quad\quad\ \ \ \le R_0^{-p}\, 2^{2(1-\theta)p}\int_{2B}\sum_{k=0}^\infty\ \int\limits_{E_k(x)}
\frac{2^{[-k\, (1-\theta)\, p]}\, R_0^{p(1-\theta)}}{\mu(B(x,2^{-k+1}R_0))}\, d\mu(y)\,d\mu(x)\\
&\quad\quad\ \ \ \le R_0^{-\theta p}\, \mu(2B)\, C_{\mathrm{D}}\, \sum_{k=-2}^\infty 2^{-kp(1-\theta)}<\infty.
\end{align*}
It follows that $u\in B^\theta_{p,p}(X)$. 
\end{proof}

A function $v$ is called a normal contraction of a function $u$ if the following holds
for all $x,y\in X$:
\[
|v(x)-v(y)|\le |u(x)-u(y)|\, \qquad \text{ and }\, \qquad
|v(x)|\le |u(x)|.
\]
Examples of normal contractions include functions $v$ of the form $v(x)=\max\{0,\, u(x)-a_0\}$ for any
non-negative number $a_0$.
In the case $a_{0} = 0$, we define $u_+(x) \coloneqq \max\{0,\, u(x)\}$. 
The following lemma is easy to check by the definition of $B^\theta_{p,p}(X)$.
Note that if $a\in\R$, $u\in B^\theta_{p,p}(X)$ and $\mu(X) < \infty$, then $u+a$ is also in $B^\theta_{p,p}(X)$.

\begin{lem}\label{lem:normal_cont}
Let $u\in B^\theta_{p,p}(X)$ and $v$ be a normal contraction of $u$.
Then $v\in B^\theta_{p,p}(X)$ and $||v||_{B^\theta_{p,p}(X)}^p\le ||u||_{B^\theta_{p,p}(X)}^p$.
As a consequence, we also have that if $u\in B^\theta_{p,p}(X)$ and $\alpha,\beta\in\R$ with $\alpha\le 0\le \beta$,  
then $w_{\alpha,\beta}:=\max\{\alpha,\, \min\{u,\, \beta\}\}$ is also in $B^\theta_{p,p}(X)$
with $||w_{\alpha,\beta}||_{B^\theta_{p,p}(X)}\le ||u||_{B^\theta_{p,p}(X)}$.
\end{lem}

The following lemma is also immediate from the definition of $B^{\theta}_{p,p}(X)$.

\begin{lem}\label{lem:leibniz}
Let $u,v\in B^\theta_{p,p}(X)\cap L^\infty(X)$. 
Then $uv\in B^\theta_{p,p}(X)$ with
\[
\norm{uv}_{B^{\theta}_{p,p}(X)} \le \norm{u}_{L^{\infty}(X)}\norm{v}_{B^{\theta}_{p,p}(X)} + \norm{v}_{L^{\infty}(X)}\norm{u}_{B^{\theta}_{p,p}(X)}.
\]
\end{lem}

\begin{lem}\label{lem:B=KS}
Suppose that $\mu$ is a doubling measure on $X$ and that $\theta>0$. 
\begin{enumerate}[\rm(1)]
	\item\label{it:B=KS} $B^\theta_{p,\infty}(X)=KS^\theta_p(X)$ as sets and as vector spaces.
	\item\label{it:besovrelation} For any $0 < \delta < \theta$, $B^{\theta}_{p,p}(X) \subset B^{\theta}_{p,\infty}(X) \subset B^{\theta - \delta}_{p,p}(X)$.
\end{enumerate}
\end{lem}

\begin{proof} 
The assertions \ref{it:B=KS} and \ref{it:besovrelation} are proved in \cite[Lemma 3.2]{Bau} and
\cite[Proposition 2.2]{GYZ} respectively, but we give the proof for the reader's convenience.

\ref{it:B=KS}:
It is direct that $B^\theta_{p,\infty}(X)\subset KS^\theta_p(X)$, and so it suffices to show the reverse inclusion.
To this end, let $u\in KS^\theta_p(X)$. Then there is some $r_u>0$ such that
\begin{equation}\label{eq:small-osc}
\sup_{0<r\le r_u}\int_X\vint{B(x,r)}\, \frac{|u(x)-u(y)|^p}{r^{\theta p}}\, d\mu(y)\, d\mu(x)\le ||u||^p_{KS^\theta_p(X)}+1.
\end{equation}
For $r>r_u$ we have that
\begin{align}\label{eq:large-osc}
\int_X&\vint{B(x,r)}\, \frac{|u(x)-u(y)|^p}{r^{\theta p}}\, d\mu(y)\, d\mu(x)\notag\\
&=\int_X \frac{\mu(B(x,r_u))}{\mu(B(x,r))}\vint{B(x,r_u)}\, \frac{|u(x)-u(y)|^p}{r^{\theta p}}\, d\mu(y)\, d\mu(x)\notag\\
 &\qquad\qquad+\int_X\frac{1}{\mu(B(x,r))}\int_{B(x,r)\setminus B(x,r_u)}\, \frac{|u(x)-u(y)|^p}{r^{\theta p}}\, d\mu(y)\, d\mu(x)\notag\\
&\le ||u||^p_{KS^\theta_p(X)}+1+\int_X\frac{2^p}{\mu(B(x,r))} \int_{B(x,r)}\frac{|u(y)|^p+|u(x)|^p}{r_u^{\theta p}}\, d\mu(y)\, d\mu(x).
\end{align} 

Note that
\begin{align*}
\int_X&\frac{2^p}{\mu(B(x,r))} \int_{B(x,r)}\frac{|u(y)|^p+|u(x)|^p}{r_u^{\theta p}}\, d\mu(y)\, d\mu(x)\\
 &=\frac{2^p}{r_u^{\theta p}}\, \int_X|u(x)|^p\, d\mu(x)
    +\frac{2^p}{r_u^{\theta p}}\, \int_X\int_X\frac{|u(y)|^p\, \chi_{B(x,r)}(y)}{\mu(B(x,r))}\, d\mu(y)\, \mu(x)\\
 &\le \frac{2^p}{r_u^{\theta p}}\, \Vert u\Vert_{L^p(X)}^p+
     \frac{2^p\, C}{r_u^{\theta p}}\, \int_X\, |u(y)|^p\, \int_X \frac{\chi_{B(y,r)}(x)}{\mu(B(y,r))}\, d\mu(x)\, d\mu(y)\\
 &=\frac{2^p(1+C)}{r_u^{\theta p}}\, \Vert u\Vert_{L^p(X)}^p,
\end{align*}
where we have used the doubling property of $\mu$ and Tonelli's theorem in the penultimate step. Now from~\eqref{eq:large-osc}
and~\eqref{eq:small-osc} above
we see that for each $r>0$ we have
\[
\int_X\vint{B(x,r)}\, \frac{|u(x)-u(y)|^p}{r^{\theta p}}\, d\mu(y)\, d\mu(x)\le
||u||^p_{KS^\theta_p(X)}+1+\frac{2^p(1+C)}{r_u^{\theta p}}\, \Vert u\Vert_{L^p(X)}^p,
\]
and as the right-hand side of the above inequality is independent of $r$, it follows that $u\in B^\theta_{p,\infty}(X)$.

\ref{it:besovrelation}:
The inclusion $B^{\theta}_{p,p}(X) \subset B^{\theta}_{p,\infty}(X)$ follows from Lemma~\ref{lem:Besov-KS} below together with claim~(1) above, and
so we prove $B^{\theta}_{p,\infty}(X) \subset B^{\theta - \delta}_{p,p}(X)$ here.
Let $u \in B^{\theta}_{p,\infty}(X)$ and fix a choice of $\alpha$ satisfying $0<\alpha < \diam(X)$.
Then we see that
\begin{align*}
	&\int_{0}^{\diam(X)}\int_{X}\vint{B(x,t)}\frac{\abs{u(x) - u(y)}^{p}}{t^{(\theta-\delta)p}}\,d\mu(y)\,d\mu(x)\,\frac{dt}{t} \\
	&= \int_{0}^{\alpha}\int_{X}\vint{B(x,t)}\frac{\abs{u(x) - u(y)}^{p}}{t^{(\theta-\delta)p}}\,d\mu(y)\,d\mu(x)\,\frac{dt}{t} \\
	&\qquad+ \int_{\alpha}^{\diam(X)}\int_{X}\vint{B(x,t)}\frac{\abs{u(x) - u(y)}^{p}}{t^{(\theta-\delta)p}}\,d\mu(y)\,d\mu(x)\,\frac{dt}{t} \\
	&\le \norm{u}_{B^{\theta}_{p,\infty}(X)}^{p}\int_{0}^{\alpha}t^{\delta p - 1}\,dt + 2^{p - 1}\biggl(\int_{\alpha}^{\diam(X)}\frac{\norm{u}_{L^p(X)}^{p}}{t^{(\theta - \delta)p + 1}}\,dt \\
	&\qquad\qquad\qquad+ \int_{\alpha}^{\diam(X)}\int_{X}\int_{X}\frac{\abs{u(y)}^{p}\chi_{B(x,t)}(y)}{t^{(\theta - \delta)p + 1}\mu(B(x,t))}\,d\mu(y)\,d\mu(x)\,dt\biggr) \\
	&\le \frac{\alpha^{\delta p}}{\delta p}\norm{u}_{B^{\theta}_{p,\infty}(X)}^{p} +
	\frac{2^{p-1}}{(\theta-\delta)p}\left[\frac{1}{\alpha^{(\theta-\delta)p}}-\frac{1}{\diam(X)^{(\theta-\delta)p}}\right] 
	\norm{u}_{L^p(X)}^{p} \\
	&\qquad+ 2^{p-1}C_{\mathrm{D}}\int_{\alpha}^{\diam(X)}\int_{X}\int_{X}\frac{\abs{u(y)}^{p}\chi_{B(y,t)}(x)}{t^{(\theta - \delta)p + 1}\mu(B(y,t))}\,d\mu(x)\,d\mu(y)\,dt \\
	&\le \frac{\alpha^{\delta p}}{\delta p}\norm{u}_{B^{\theta}_{p,\infty}(X)}^{p} +
	\frac{2^{p-1}\, (1+C_D)}{(\theta-\delta)p}\left[\frac{1}{\alpha^{(\theta-\delta)p}}-\frac{1}{\diam(X)^{(\theta-\delta)p}}\right] 
	\norm{u}_{L^p(X)}^{p},
\end{align*}
where we have used the doubling property of $\mu$ and Tonelli's theorem in the third inequality.
Note if $X$ is unbounded, then $\frac{1}{\diam(X)^{(\theta-\delta)p}} = 0$.
This estimate shows that $u \in B^{\theta - \delta}_{p,p}(X)$.
\end{proof}

In general, unlike the energy related to $B^\theta_{p,\infty}(X)$,
the energy $\norm{u}_{KS^{\theta}_{p}(X)}$ is zero whenever $u \in B^{\theta}_{p,p}(X)$.

\begin{lem}\label{lem:Besov-KS}
Let $\mu$ be a doubling measure on $X$ and $\theta>0$. Then 
$B^\theta_{p,p}(X)\subset KS^\theta_p(X)$
with $\Vert u\Vert_{KS^\theta_p(X)}=0$ whenever
$u\in B^\theta_{p,p}(X)$.
\end{lem}

\begin{proof}
Let $u\in B^\theta_{p,p}(X)$. Then we have that
\[
\int_0^{\diam X} \int_X\vint{B(x,t)}\frac{|u(y)-u(x)|^p}{t^{\theta p}}\, d\mu(y)\, d\mu(x)\, \frac{dt}{t}<\infty.
\]
For $t>0$ we set
\[
\calE_\theta(u,t):=\int_X\vint{B(x,t)}\frac{|u(y)-u(x)|^p}{t^{\theta p}}\, d\mu(y)\, d\mu(x).
\]
Let $k_*\in \mathbb Z\cup\{\infty\}$ be the maximum of all the positive integers $k$ such that $2^{k-1}<\diam X$.
By the doubling property of $\mu$ we have
\begin{align*}
\int_0^{\diam X} \int_X\vint{B(x,t)}\frac{|u(y)-u(x)|^p}{t^{\theta p}}\, d\mu(y)\, d\mu(x)\, \frac{dt}{t}
 &\quad\ge 
 \sum_{i=-\infty}^{k_*-2}
\int_{2^i}^{2^{i+1}}\, \calE_\theta(u,t)\, \frac{dt}{t}\\ 
&\quad \approx\, \sum_{i=-\infty}^{k_*-2}\calE_\theta(u,2^i).
\end{align*}
Since the left-most expression is finite, it follows that the series on the right-hand side of the
above estimate is also finite, and therefore
\[
\lim_{i\to-\infty}\calE_\theta(u,2^i)=0.
\]
By the doubling property of $\mu$ we also have that for positive real numbers $t<\diam(X)$,
\[
\frac{1}{C}\, \calE_\theta(u,2^{i-1})\le \calE_\theta(u,t)\le C\, \calE_\theta(u,2^i)\text{ whenever }2^{i-1}\le t\le 2^i.
\]
It follows that
\[
\limsup_{t\to 0^+}\calE_\theta(u,t)\le C\, \lim_{i\to-\infty}\calE_\theta(u,2^i)=0,
\]
completing the proof.
\end{proof}

\section{Examples}

The following examples show that even though the two vector spaces considered in Lemma~\ref{lem:Besov-KS}
are the same as sets, their energy norms
can be incomparable.

\begin{example}\label{ex:bow-tie}
In this example we consider $X$ to be the union of two $n$-dimensional hypercubes
glued at the vertex $o=(0,\cdots,0)$, given by
\[
X=[0,1]^n\, \bigcup\, [-1,0]^n,
\]
equipped with the Euclidean metric and the $n$-dimensional Lebesgue measure $\mathcal{L}^{n}$.
Here, with $u:=\chi_E$ where $E=[0,1]^n$,
we see that $u\in B^\theta_{p,p}(X)$ precisely when $p\theta<n$, but %from Lemma~\ref{lem:Besov-KS} 
we have $\Vert u\Vert_{B^\theta_{p,\infty}(X)}>0$ (see \eqref{e:bowtie.KS} for a detailed calculation) but from Lemma~\ref{lem:Besov-KS} we also have that
$\Vert u\Vert_{KS^\theta_p(X)}=0$. To see that $u\in B^\theta_{p,p}(X)$ when $p\theta<n$,
we decompose the two pieces $E$ and $X\setminus E$ into dyadic annuli given by
$L_i:=\{(x_{1},\dots,x_{n})\in E\, :\, 2^{-i-1}R<\sqrt{x_{1}^2+ \cdots + x_{n}^2}\le 2^{-i}R\}$ and
$R_i=\{(x,y)\in X\setminus E\, :\, 2^{-i-1}R<\sqrt{x_{1}^2+ \cdots + x_{n}^2}\le 2^{-i}R\}$ with $R=\sqrt{n}$, we have that
\begin{align*}
\int_X\int_X\frac{|\chi_E(x)-\chi_E(y)|^p}{d(x,y)^{n+\theta p}}&\, d\mathcal{L}^n(y)\, d\mathcal{L}^n(x)\\
\approx &\sum_{i,j\in\mathbb{N}\cup\{0\}}\int_{L_i}\int_{R_j}\, \frac{|\chi_E(x)-\chi_E(y)|^p}{d(x,y)^{n+\theta p}}\, d\mathcal{L}^n(y)\, d\mathcal{L}^n(x)\\
\approx&\sum_{i=0}^\infty\, \sum_{j=i}^\infty\, \int_{L_i}\int_{R_j}\, \frac{1}{d(x,y)^{n+\theta p}}\, d\mathcal{L}^n(y)\, d\mathcal{L}^n(x)\\ 
\approx&\sum_{i=0}^\infty\sum_{j=i}^\infty\frac{2^{-ni}R^n\, 2^{-nj}R^n}{(2^{-i}+2^{-j})^{n+\theta p}\,R^{n+\theta p}}\\
\approx&\sum_{i=0}^\infty\sum_{j=i}^\infty\, 2^{i\theta p}\, 2^{-nj}
\approx \sum_{i=0}^\infty 2^{-i(n-\theta p)}.
\end{align*} 
The above sum is finite if and only if $\theta p<n$. Thus $\chi_E\in B^\theta_{p,p}(X)$ if and only if $\theta p<n$, and so
$\chi_E\in KS^\theta_p(X)$ with $\norm{u}_{KS^{\theta}_{p}(X)} = 0$ whenever $\theta p<n$.

\begin{figure}[tb]
\includegraphics[height=150pt]{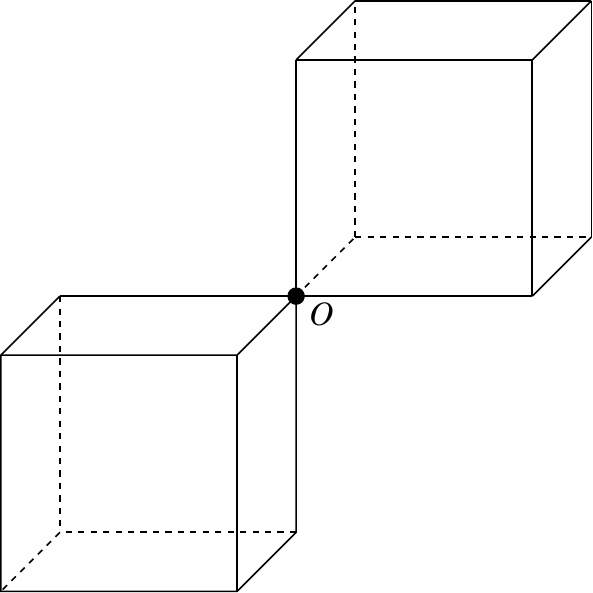}
\caption{Gluing of two unit cubes at the origin}\label{fig:bowtie}
\end{figure}

In addition, 
{in computing $\vint{B(x,r)}\frac{\abs{\chi_{E}(x) - \chi_{E}(y)}^{p}}{r^{p\theta}}\,d\mathcal{L}^{n}(y)$ for $x\in E$,
we need only consider $x=(x_1,\cdots,x_n)\in E$ for which $\sqrt{x_1^2+\cdots+x_n^2}<r$, and so by restricting our attention
to the slices $L_j$ for which $2^{-j}R\lesssim r$, we obtain}
\begin{equation}\label{e:bowtie.KS}
	\int_{X}\vint{B(x,r)}\frac{\abs{\chi_{E}(x) - \chi_{E}(y)}^{p}}{r^{p\theta}}\,d\mathcal{L}^{n}(y)\,d\mathcal{L}^{n}(x)
	\approx r^{n - p\theta}.
\end{equation}
Hence $\chi_{E} \in KS^{\theta}_{p}(X)$ whenever $p\theta \le n$;
note that $\norm{u}_{KS^{\theta}_{p}(X)} = 0$ if $p\theta < n$. 

The following proposition states a relation between $KS^{1}_{n}(X)$ and $N^{1,n}(X)$.
Set $E_{1} \coloneqq [0,1]^{n}$, $E_{2} \coloneqq [-1,0]^{n}$ and $o \coloneqq (0,\dots,0) \in E_{1} \cap E_{2}$ for simplicity.
In what follows, if $u$ is a function defined on a set $E\subset X$, then the zero-extension of $u$ to $X\setminus E$
is denoted by $u\chi_E$. 

\begin{prop}\label{prop:KSdiff.bowtie} In the above setting $X=[0,1]^n\cup[-1,0]^n$, it follows that
	\begin{enumerate}[\rm(1)]
		\item\label{it:determine.bowtie}
		\begin{equation*}
			KS^{1}_{n}(X)
			= \biggl\{ u_{1}\chi_{E_{1}} + u_{2}\chi_{E_{2}} \biggm| 
				 u_{i} \in N^{1,n}(E_{i}), i \in \{ 1,2 \},\
				I_{KS}(u_{1},u_{2}) < \infty 
			\biggr\}, 
		\end{equation*}
		where
		\begin{equation*}
			I_{KS}(u_{1},u_{2}) \coloneqq \limsup_{r \to 0^+}\int_{E_{1} \cap B(o,r)}\int_{E_{2} \cap B(o,r)}\frac{\abs{u_{1}(x) - u_{2}(y)}^{n}}{r^{2n}}\,d\mathcal{L}^{n}(y)\,d\mathcal{L}^{n}(x).
		\end{equation*}
		\item\label{it:relation.bowtie} $KS^{1}_{n}(X) \subsetneq N^{1,n}(X)$.
	\end{enumerate}
\end{prop}
\begin{proof}
	We first note that the $n$-modulus of the all rectifiable curves in $X$ through $o$ is $0$
by \cite[Corollary 5.3.11]{HKSTbook}, and that $KS^{1}_{n}(X) \subset N^{1,n}(X)$
by \cite[Theorem 10.5.1]{HKSTbook} and \cite[Corollary 6.5]{LPZ}.
As a consequence, we have
	\[
	N^{1,n}(X) = \bigl\{ u_{1}\chi_{E_{1}} + u_{2}\chi_{E_{2}} \bigm| u_{i} \in N^{1,n}(E_{i})\text{ for }i=1,2 \bigr\}.
	\]
	In addition, $KS^{1}_{n}(E_{i}) = N^{1,n}(E_{i})$ with comparable norms
	by \cite[Theorem~10.5.2]{HKSTbook}. When $u\in KS^1_n(X)$, necessarily
	$u\chi_{E_i}\in KS^1_n(E_i)$. This is because when $x\in E_i$ and $0<r<1$, we must have
	that $\mathcal{L}^n(B(x,r))\approx r^n\approx\mathcal{L}^n(B(x,r)\cap E_i)$.

	{\bf Proof of~(1):} Let $u_{i} \in N^{1,n}(E_{i})$ for $i = 1,2$, and set $u=u_1\chi_{E_1}+u_2\chi_{E_2}$.
	We define
	\begin{equation*}
		\mathcal{E}_{r}^{KS}(v; A_{1}, A_{2})
		\coloneqq \int_{A_{1}}\int_{A_{2} \cap B(x,r)}\frac{\abs{v(x) - v(y)}^{n}}{r^{n}}\,d\mathcal{L}^{n}(y)\,d\mathcal{L}^{n}(x),
	\end{equation*}
	for $v \in L^{n}(A_1\cup A_2)$ and Borel sets $A_{i}$ of $X$.
	Observe that
	\begin{align*}
		\int_{X}\vint{B(x,r)}&\frac{\abs{u(x) - u(y)}^{n}}{r^{n}}\,d\mathcal{L}^{n}(y)\,d\mathcal{L}^{n}(x) \\
		&\approx \frac{1}{r^{n}}\Bigl(\mathcal{E}_{r}^{KS}(u_{1}; E_{1}, E_{1}) + \mathcal{E}_{r}^{KS}(u_{2}; E_{2}, E_{2}) \\
		&\qquad\qquad+ \mathcal{E}_{r}^{KS}(u; E_{1}, E_{2}) + \mathcal{E}_{r}^{KS}(u; E_{2}, E_{1})\Bigr).
	\end{align*}
	Since
	\[
	\limsup_{r \to 0^+}\frac{\mathcal{E}_{r}^{KS}(u_{i}; E_{i}, E_{i})}{r^{n}}
	\approx \int_{E_i}|\nabla u_i(x)|^n\, d\mathcal{L}^{n}(x)
	\] 
	it suffices to prove that $u \in KS^{1}_{n}(X)$ if and only if
$I_{KS}(u_{1},u_{2}) < \infty$.

Given the above discussion, we know that $u\in KS^1_n(X)$ if and only if
\begin{equation}\label{eq:u1-vs-u2}
\limsup_{r\to 0^+}\, \frac{1}{r^{n}}\, \Bigl(\mathcal{E}_{r}^{KS}(u; E_{1}, E_{2}) + \mathcal{E}_{r}^{KS}(u; E_{2}, E_{1})\Bigr)<\infty.
\end{equation}
Let us focus our attention on $\mathcal{E}_{r}^{KS}(u; E_{1}, E_{2})$, with the second term above being handled in a similar manner.
Note that
\[
\mathcal{E}_{r}^{KS}(u; E_{1}, E_{2})=\int_{E_1}\, \int_{E_2\cap B(x,r)}\frac{|u_1(x)-u_2(y)|^n}{r^n}\, d\mathcal{L}^{n}(y)\,d\mathcal{L}^{n}(x),
\]
and so in order for $E_2\cap B(x,r)$ to be non-empty when $x\in E_1$, it must be the case that $x\in B(o,r)$. Thus
\begin{align*}
\mathcal{E}_{r}^{KS}(u; E_{1}, E_{2})&=\int_{E_1\cap B(o,r)}\, \int_{E_2\cap B(x,r)}\frac{|u_1(x)-u_2(y)|^n}{r^n}\, d\mathcal{L}^{n}(y)\, d\mathcal{L}^{n}(x)\\
& \le \int_{E_1\cap B(o,r)}\, \int_{E_2\cap B(o,r)}\frac{|u_1(x)-u_2(y)|^n}{r^n}\, d\mathcal{L}^{n}(y)\, d\mathcal{L}^{n}(x),
\end{align*}
and moreover,
\begin{align*}
\mathcal{E}_{r}^{KS}(u; E_{1}, E_{2})&=\int_{E_1\cap B(o,r)}\, \int_{E_2\cap B(x,r)}\frac{|u_1(x)-u_2(y)|^n}{r^n}\, d\mathcal{L}^{n}(y)\, d\mathcal{L}^{n}(x)\\
& \ge \int_{E_1\cap B(o,r/4)}\, \int_{E_2\cap B(o,r/4)}\frac{|u_1(x)-u_2(y)|^n}{r^n}\, d\mathcal{L}^{n}(y)\, d\mathcal{L}^{n}(x).
\end{align*}
Similarly, we also see that
\begin{align*}
\mathcal{E}_{r}^{KS}(u; E_{2}, E_{1})&\le \int_{E_1\cap B(o,r)}\, \int_{E_2\cap B(o,r)}\frac{|u_1(x)-u_2(y)|^n}{r^n}\, d\mathcal{L}^{n}(y)\, d\mathcal{L}^{n}(x),\\
\mathcal{E}_{r}^{KS}(u; E_{2}, E_{1})&\ge \int_{E_1\cap B(o,r/4)}\, \int_{E_2\cap B(o,r/4)}\frac{|u_1(x)-u_2(y)|^n}{r^n}\, d\mathcal{L}^{n}(y)\, d\mathcal{L}^{n}(x).
\end{align*}
It follows that~\eqref{eq:u1-vs-u2} holds if and only if
\begin{align*}
	&I_{KS}(u_1,u_2) \\
	&= \limsup_{r\to 0^+}\int_{E_1\cap B(o,r)}\, \int_{E_2\cap B(o,r)}\frac{|u_1(x)-u_2(y)|^n}{r^{2n}}\, d\mathcal{L}^{n}(y)\, d\mathcal{L}^{n}(x) 
	< \infty. 
\end{align*} 
	These complete the proof of~(1).

	{\bf Proof of~(2):}  It suffices to find $u \in N^{1,n}(X) \setminus KS^{1}_{n}(X)$;  note that
	$u\in N^{1,n}(X)$ if and only if $u\vert_{E_i}\in N^{1,n}(E_i)$ for $i=1,2$. 
	By direct computation or by~\cite{GS}, we know that the function $v(x) \coloneqq \log{(-\log{\abs{x}})}$ for $x \in E_{1} \setminus \{ o \}$
		belongs to $N^{1,n}(E_{1})$. Note that
		\[
		\lim_{r\to 0^+}\essinf_{E_{1} \cap B(o,r)}\abs{v}=\infty.
		\] 
	Now we define $u \in N^{1,n}(X)$ by $u(x) \coloneqq v(x)$ for $x \in E_{1}$ and $u(x) \coloneqq 0$ for $x \in E_{2} \setminus \{ o \}$.
	Then we easily see that
	\[
	\vint{E_{1} \cap B(o,r)}\vint{E_{2} \cap B(o,r)}\abs{u(x) - u(y)}^{n}\,d\mathcal{L}^{n}(y)\,d\mathcal{L}^{n}(x)
	\ge \biggl(\essinf_{E_{1} \cap B(o,r)}\abs{v}\biggr)^{n},  
	\]
	and so  
	$u \not\in KS^{1}_{n}(X)$ though 
	$u \in N^{1,n}(X)$, since $\essinf_{E_{1} \cap B(o,r)}\abs{v}\to\infty$
	as $r\to 0^+$.  
\end{proof}

Note that the dimension of $B^1_{p,p}(X)$ is $2$ when $1<p<n$ .
Moreover, thanks to~\cite{BBM}
applied to each of the two $n$-dimensional hypercubes of $X$ and \eqref{e:bowtie.KS}, we know that $\theta_p(X)=n/p$, in particular, $\theta_{p}(X) > 1$ when $1<p<n$.

\begin{proof}[Proof of Theorem \ref{thm:critical.gluedSC} for the glued hypercubes]
Note that $\hdim = n$ and $\pwalk = p$ in this case. 
As already mentioned, $\theta_p(X)=n/p = \hdim/p$ when $p<n$. 
The estimate \eqref{e:bowtie.KS}, along with the fact that $\theta_{p}([0,1]^{n}) = 1$, shows $\theta_p(X)= 1 = \pwalk/p$ when $p \ge n$. 
Moreover, for $B^\theta_{p,p}(X)$ to be dense
in $L^p(X)$ it is necessary to have that $B^\theta_{p,p}([0,1]^n)$ be dense in $L^p([0,1]^n)$, and this requires $\theta<1$. It follows that $\theta_p^*(X)\le 1$. On the other hand, when $\theta<1$,  
$B^\theta_{p,p}(X)$ is dense in $L^p(X)$ 
due to the results of~\cite{BBS} because the class of Lipschitz continuous functions forms a dense subclass of  both spaces.
Hence we have $\theta_p^*(X)=1=\pwalk/p$.
\end{proof} 
\end{example}

A similar example can be considered by gluing two copies of the Sierpi\'nski gasket, but the resultant example has dramatically different
phenomena in comparison to Example~\ref{ex:bow-tie} above. 
Precisely, for any $p \in (1,\infty)$, 
$\theta_{p}(X)=\theta_{p}^{\ast}(X)$ 
%the critical exponents $\theta_{p}(X)$ and $\theta_{p}^{\ast}(X)$ coincide with each other 
for this example as shown in Theorem \ref{thm:critical.gluedSG} below. 
In comparison, in Example~\ref{ex:bow-tie} we have that $\theta_p(X)=n/p$. However,
when $\theta\ge 1$ we necessarily have that any function $u\in B^\theta_{p,p}(X)$ must be constant on each 
of the two cubes $[0,1]^n$ and $[-1,0]^n$, thanks to the results in~\cite{BBM}.
 Therefore $\theta_p^*(X)=1$ in Example~\ref{ex:bow-tie}.

\begin{example}[Gluing copies of the Sierpi\'{n}ski gasket]\label{ex:glueSG}
	\begin{figure}[tb]\centering
		\includegraphics[height=150pt]{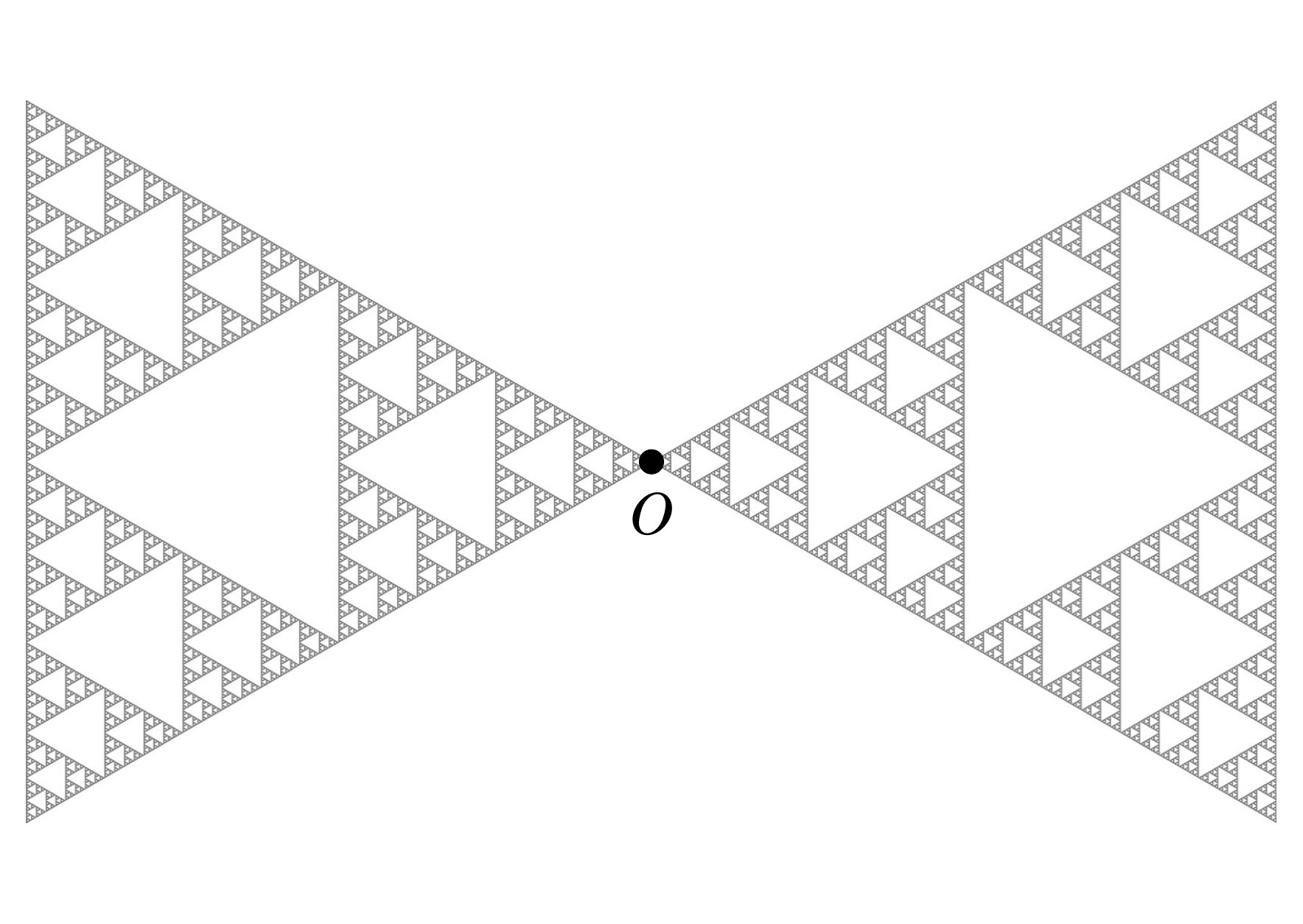} 
		\caption{Gluing of two copies of the Sierpi\'{n}ski gasket}\label{fig:gluedSG}
	\end{figure} 
	In this example, we consider $X$ to be the union of
		two copies of
		the $n$-dimensional standard Sierpi\'{n}ski gasket glued at a point. Let $n \in \mathbb{N}$ with $n \ge 2$, let $K$ be the
		standard $n$-dimensional Sierpi\'nski gasket, rotated so that it is symmetric about the $x_n$-axis in $\mathbb{R}^n$ and located in the half-space $\{ x_n\ge 0 \}$
		and has a vertex at $o \coloneqq (0,0,\cdots,0)$, $K^+ \coloneqq K$ and $K^-$ the reflection of $K$ in the hyperplane $\{ x_n=0 \}$, and then set $X=K^+\cup K^-$ (see Figure \ref{fig:gluedSG} for the case $n = 2$). 
    Let $d$ be the Euclidean metric (restricted to $X$) and
    $\mu$ be the $\hdim$-dimensional Hausdorff measure on $X$, where $\hdim \coloneqq \log{(n+1)}/\log{2}$.
    Then $\mu$ is Ahlfors $\hdim$-regular on $X$, i.e., there exists $c_{1}\ge 1$ such that
    \begin{equation}\label{e:AR}
    	c_{1}^{-1}r^{\hdim} \le \mu(B(x,r)) \le c_{1}r^{\hdim} \quad \text{for any } x\in X, \ \ \ 0<r<\diam(X).
    \end{equation}
    Now let us focus on the following Besov-type energy functional of $\chi_{K^{+}}$:
    \[
    \int_{X}\vint{B(x,r)}\frac{\abs{\chi_{K^{+}}(x) - \chi_{K^{+}}(y)}^{p}}{r^{p\theta}}\, d\mu(y)\,d\mu(x), \quad r > 0.
    \]
    Note that if $x \in K^{-}$ and $B(x,r) \cap K^{+} \neq \emptyset$, then $o \in B(x,r)$ and hence $B(x,r) \subset B(o,2r)$.
    Therefore,
    \begin{align}\label{e:glueSG.KSupper}
        &\int_{X}\vint{B(x,r)}\frac{\abs{\chi_{K^{+}}(x) - \chi_{K^{+}}(y)}^{p}}{r^{p\theta}}\,\mu(dy)\,\mu(dx) \nonumber \\
        &\le c_{1}\,r^{-\hdim}\int_{B(o,2r) \cap K^{-}}\int_{B(o,2r) \cap K^{+}}\frac{\abs{\chi_{K^{+}}(x) - \chi_{K^{+}}(y)}^{p}}{r^{p\theta}}\,\mu(dy)\,\mu(dx) \nonumber \\
        &\le c_{1}\,r^{-\hdim - p\theta}\mu(B(o,2r))^{2}
        \le c_{1}^{3}\,r^{\hdim - p\theta }.
    \end{align} 
    Since $\mu(B(o,r/4) \cap K^{\pm}) \ge c_{2}r^{\hdim}$, we also have
    \begin{align}\label{e:glueSG.KSlower}
        &\int_{X}\vint{B(x,r)}\frac{\abs{\chi_{K^{+}}(x) - \chi_{K^{+}}(y)}^{p}}{r^{p\theta}}\,\mu(dy)\,\mu(dx) \nonumber \\
        &\ge c_{1}^{-1}r^{-\hdim}\int_{B(o,r/4) \cap K^{-}}\int_{B(o,r/4) \cap K^{+}}\frac{\abs{\chi_{K^{+}}(x) - \chi_{K^{+}}(y)}^{p}}{r^{p\theta}}\,\mu(dy)\,\mu(dx) \nonumber \\
        &\ge c_{1}r^{-\hdim - p\theta}\mu(B(o,r/4) \cap K^{-})\mu(B(o,r/4) \cap K^{+})
        \ge c_{1}^{-1}c_{2}^{2}\,r^{\hdim - p\theta}.
    \end{align}
    Hence $\chi_{K^{+}} \in B^{\theta}_{p,p}(X)$ if and only if $0<\theta<\hdim/p$, and 
    $\chi_{K^{+}} \in KS^{\theta}_{p}(X)$ if and only if $0<\theta \le\hdim/p$.
    Moreover, $\norm{\chi_{K^+}}_{KS^{\theta}_{p}(X)} = 0$ for $\theta \in (0,\hdim/p)$, and $\norm{\chi_{K^+}}_{KS^{\hdim/p}_{p}(X)} > 0$.
    In particular, the $p$-energy form $(\norm{\,\cdot\,}_{KS^{\theta}_{p}(X)}^{p}, KS^{\theta}_{p}(X))$ is reducible when $\theta \in (0,\hdim/p)$.

    Let $\pwalk$ be the $p$-walk dimension of the $n$-dimensional standard Sierpi\'nski gasket $K^+$, i.e., 
    $\pwalk = \log{((n + 1)\rho_{p})}/\log{2}$ where $\rho_{p}$ is the $p$-scaling factor of
    $K^+$ used in constructing the analog of the Sobolev space $\mathcal{F}_p$ on the gasket
    (see \cite[Subsection 9.2]{KS1} for further details on the $p$-walk dimension of Sierpi\'nski gaskets).
    From \cite[Theorems 5.16, 5.26, Corollary 5.27, Proposition 5.28]{KS2} and
    Lemma~\ref{lem:B=KS}\ref{it:besovrelation} above, we know that
    $\theta_{p}(K^{\pm}) = \theta_{p}^{\ast}(K^{\pm}) = \pwalk/p$. 
    It is known that $\pwalk > p$ and $\pwalk > \hdim$ for any $p \in (1,\infty)$; see \cite[Theorems~9.13, B.8 and (8.39)]{KS1}
    and \cite[Proposition 3.3]{Kig}. In the next theorem we determine $\theta_{p}(X)$ and $\theta_{p}^{\ast}(X)$ (note that the Ahlfors regular conformal dimension of the $n$-dimensional standard Sierpi\'{n}ski gasket is $1$;
    see, e.g.,~\cite[Theorem~B.8]{KS1}).

    \begin{thm}\label{thm:critical.gluedSG}
    	In the above setting of $X = K^{+} \cup K^{-}$, where each $K^{\pm}$ is the $n$-dimensional
	Sierpi\'nski gasket, we have  
	$\theta_{p}(X) = \theta_{p}^{\ast}(X) = \frac{\pwalk}{p}$ for $1<p<\infty$.
	\end{thm}

	\begin{proof}
		We first show that $\theta_{p}(X) = \pwalk/p$.
		Since $B_{p,\infty}^{\pwalk/p}(K^{\pm}) \subset C(K^{\pm})$ and
		$B_{p,\infty}^{\pwalk/p}(K^{\pm})$ is dense in $C(K^{\pm})$ by~\cite[Corollary 9.11]{KS1}
		and~\cite[Theorem 5.26]{KS2}, we have $\theta_{p}(X) \ge \pwalk/p$. 
		Indeed, by this density we can find a non-constant function $u\in B_{p,\infty}^{\pwalk/p}(K^+)$,
		and then its reflection $v$ given by
		\[
		 v(x)=\begin{cases} u(x) &\text{ if }x\in K^+,\\
		    u(-x) &\text{ if }x\in K^-, \end{cases}
		\]
		belongs to $B_{p,\infty}^{\pwalk/p}(X)$, and so we have a non-constant function in $B_{p,\infty}^{\pwalk/p}(X)$.

		For any $\theta > \pwalk/p$ and $u \in B^{\theta}_{p,p}(X)$, we have from
		Lemma~\ref{lem:B=KS}~\ref{it:besovrelation} that $u|_{K^{\pm}} \in B^{\theta}_{p,\infty}(K^{\pm})$.
		Then $u|_{K^{+}}$ and $u|_{K^{-}}$ must be constant functions since $\theta_{p}(K^{\pm}) = \pwalk/p$. 
		Since $\chi_{K^{+}} \not\in B^{\theta}_{p,p}(X)$ by the discussion preceding the statement of the theorem being proved here,
		and since $\theta > \pwalk/p > \hdim/p$, the function $u$ has to be
		constant on $X$. Hence, $\theta_{p}(X) \le \pwalk/p$.
		The proof of $\theta_{p}(X) = \pwalk/p$ is completed.

		Next we prove that $\theta_{p}^{\ast}(X) = \pwalk/p$.
		It suffices to show that $B^{\pwalk/p}_{p,\infty}(X)$ is dense in $C(X)$; indeed, if this is true, then we have
		from Lemma~\ref{lem:B=KS}~\ref{it:besovrelation}
		and the fact that $C(X)$ is dense in $L^p(X)$
		that $B^{\theta}_{p,p}(X)$ is dense in
		$L^{p}(X)$ for any
		$\theta < \pwalk/p$ and hence $\theta_{p}^{\ast}(X) \ge \pwalk/p$. (Recall that
		$\theta_{p}^{\ast}(X) \le \theta_{p}(X) = \pwalk/p$.) 

		To show that $B^{\pwalk/p}_{p,\infty}(X)$ is dense in $C(X)$, let $u \in C(X)$. 
		We can assume that $u(o) = 0$ by adding a constant function.
		Recall that $u_{+}(x) \coloneqq \max\{ 0, u(x) \}$ and set $u_{-} \coloneqq u_{+} - u$.
		Since $B^{\pwalk/p}_{p,\infty}(K^{\pm})$ is dense in $C(K^{\pm})$, for any $\varepsilon > 0$ there exist four continuous functions
		$u^{K^{+}}_{\pm,\varepsilon} \in {B^{\pwalk/p}_{p,\infty}(K^{+})}$,
		$u^{K^{-}}_{\pm,\varepsilon} \in {B^{\pwalk/p}_{p,\infty}(K^{-})}$
		such that 
		\[
		\sup_{x \in K^{+}}\abs{u_{\pm}(x) - u_{\pm,\varepsilon}^{K^{+}}(x)} \le \varepsilon, \ \text{ and }\
		\sup_{x \in K^{-}}\abs{u_{\pm}(x) - u_{\pm,\varepsilon}^{K^{-}}(x)} \le \varepsilon.
		\]
		We can also assume that $u_{\pm,\eps}^{K^+}$ and $u_{\pm,\eps}^{K^-}$
		are nonnegative.
		Since $u(o)=0$ and $u_{\pm,\eps}^{K^+}, u_{\pm,\eps}^{K^-}$ are continuous, there exists $\delta > 0$ such that
		\[
		\sup_{x \in B(o,\delta) \cap K^{+}}\abs{u_{\pm,\varepsilon}^{K^{+}}(x)} \le 2\varepsilon\ \text{ and }\
		\sup_{x \in B(o,\delta) \cap K^{-}}\abs{u_{\pm,\varepsilon}^{K^{-}}(x)} \le 2\varepsilon.
		\]
		Now we set
		\[
		u_{\varepsilon} \coloneqq \bigl[(u_{+,\varepsilon}^{K^{+}} - 2\varepsilon)_{+}  - (u_{-,\varepsilon}^{K^{+}} - 2\varepsilon)_{+}\bigr]\chi_{K^{+}} + \bigl[(u_{+,\varepsilon}^{K^{-}} - 2\varepsilon)_{+}  - (u_{-,\varepsilon}^{K^{-}} - 2\varepsilon)_{+}\bigr]\chi_{K^{-}}.
		\]
		Then $u_\eps\in C(X)$.
		Note that $u_{\varepsilon} = 0$ on $B(o,\delta)$ and that $\norm{u - u_{\varepsilon}}_{\sup} \le 3\varepsilon$.
		We conclude that $u_{\varepsilon} \in B^{\pwalk/p}_{p,\infty}(X)$ by using the ``locality'' of
		$\norm{\,\cdot\,}_{KS^{\pwalk/p}_{p}(X)}$; indeed,
		\[
		\norm{u_{\varepsilon}}_{KS^{\pwalk/p}_{p}(X)}^{p}
		\le \norm{u_{\varepsilon}|_{K^{+}}}_{KS^{\pwalk/p}_{p}(K^{+})}^{p}
		+ \norm{u_{\varepsilon}|_{K^{-}}}_{KS^{\pwalk/p}_{p}(K^{-})}^{p}.
		\]
		Therefore, $B^{\pwalk/p}_{p,\infty}(X)$ is dense in $C(X)$.
	\end{proof}
\end{example}

\begin{example}[Gluing copies of the Sierpi\'{n}ski carpet]\label{ex:glueSC}
In this example, we consider $X$ to be the union of
	two isometric copies of
	the planar standard Sierpi\'{n}ski carpet glued at a point. We confine ourselves to the planar case unlike in
	Examples~\ref{ex:bow-tie} and \ref{ex:glueSG}, because the construction of a self-similar $p$-energy form and its corresponding Sobolev analog $\mathcal{F}_p$ for all $1<p<\infty$ including the case where $p$ is less than or equal to the Ahlfors regular conformal dimension (denoted by $d_{\mathrm{ARC}}$ below)
	is currently known only for the planar carpet.

	Let $K$ be the
		standard Sierpi\'nski carpet, rotated so that it is symmetric about the line $\{ y=x \}$ in
		$\mathbb{R}^2$ and located in the quadrant $\{ x \le 0, y \le 0 \}$
		and has a vertex at $o \coloneqq (0,0)$, $K^+ \coloneqq K$ and $K^-$ be the reflection of
		$K$ in the line $\{ y=-x \}$, and then set $X=K^+\cup K^-$ (see Figure \ref{fig:gluedSC}). 
    Let $d$ be the Euclidean metric (restricted on $X$) and $\mu$ be the $\hdim$-dimensional Hausdorff measure on $X$,
    where $\hdim \coloneqq \log{8}/\log{3}$.
    Then $\mu$ is Ahlfors $\hdim$-regular on $X$, i.e., \eqref{e:AR} holds.
    Similar to \eqref{e:glueSG.KSupper} and \eqref{e:glueSG.KSlower}, we can estimate
    \begin{equation}\label{e:gluedSC.KS}
    	\int_{X}\vint{B(x,r)}\frac{\abs{\chi_{K^{+}}(x) - \chi_{K^{+}}(y)}^{p}}{r^{p\theta}}\,\mu(dy)\,\mu(dx) \approx r^{\hdim - p\theta}. 
    \end{equation}
    Hence $\chi_{K^{+}} \in B_{p,p}^{\theta}(X)$ if and only if $\theta \in (0,\hdim/p)$, and $\chi_{K^{+}} \in KS^{\theta}_{p}(X)$
    if and only if $\theta \in (0,\hdim/p]$.
    Also, we have $\norm{\chi_{K^{+}}}_{KS^{\theta}_{p}(X)} = 0$ for $\theta \in (0,\hdim/p)$ and
    $\norm{\chi_{K^{+}}}_{KS^{\hdim/p}_{p}(X)} > 0$.
    In particular, $(\norm{\,\cdot\,}_{KS^{\theta}_{p}(X)}^{p}, KS^{\theta}_{p}(X))$ is reducible when $\theta \in (0,\hdim/p)$.
    \begin{figure}[tb]\centering
		\includegraphics[height=150pt]{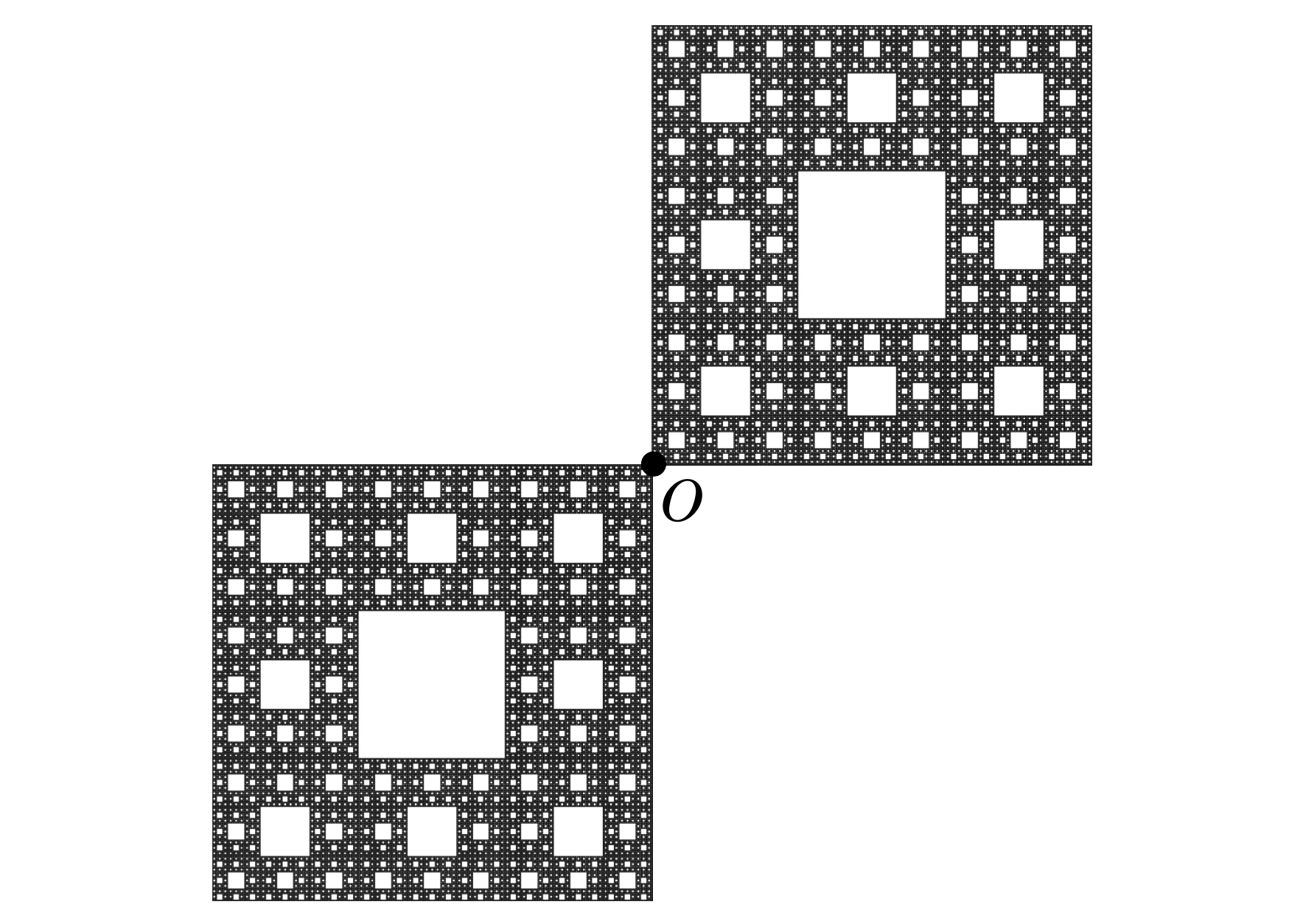} 
		\caption{Gluing of two copies of the Sierpi\'{n}ski carpet}\label{fig:gluedSC}
	\end{figure}

    Similar to Example \ref{ex:glueSG}, from \cite[Theorems 1.1~(ii), 1.4 and proof of Theorem 7.1]{MS} and Lemma~\ref{lem:B=KS}~\ref{it:besovrelation}, we know that $\theta_{p}(K^{\pm}) = \theta_{p}^{\ast}(K^{\pm}) = \pwalk/p$ where $\pwalk$ is the $p$-walk dimension of the Sierpi\'{n}ski carpet.
    By \cite[Theorem 2.24]{Shi} or \cite[Theorem 9.8]{KS1}, we have $\pwalk > p$ for any $p \in (1,\infty)$.
    Next let us recall
    a relation with the \emph{Ahlfors regular conformal dimension} $d_{\mathrm{ARC}}$ of the Sierpi\'{n}ski
    carpet that is discussed in the end of introduction.
    From \cite[Corollary 3.7]{BK} and \cite[Corollary 1.4]{CP} (see also \cite[Proof of Proposition 1.7]{CC}),
    we know that $\pwalk > \hdim$ if and only if $p > d_{\mathrm{ARC}}$, that $\pwalk < \hdim$ if and only if
    $p < d_{\mathrm{ARC}}$, and that $\pwalk = \hdim$ for $p = d_{\mathrm{ARC}}$.
    Also, $d_{\mathrm{ARC}} \ge 1 + \frac{\log{2}}{\log{3}}$ by \cite[Remark 1]{BT}.
    We can determine $\theta_{p}(X)$ and $\theta_{p}^{\ast}(X)$ as in Theorem \ref{thm:critical.gluedSC},
    in particular, there is a gap between $\theta_{p}(X)$ and $\theta_{p}^{\ast}(X)$ when $1<p <d_{\mathrm{ARC}}$. 

\begin{proof}[Proof of Theorem \ref{thm:critical.gluedSC} for the glued Sierpi\'nski carpets] 
%We first consider the case that $X$ is the gluing of two copies of the $n$-dimensional Euclidean cube
%at a vertex, that is, $X = [0,-1]^{n} \cup [0,1]^{n}$. Then by~\eqref{e:bowtie.KS} we know that when $p<n$,
%$\theta_p(X)=n/p$; note that when $p<n$ we have $\pwalk=p$. Moreover, for $B^\theta_{p,p}(X)$ to be dense
%in $L^p(X)$ it is necessary to have that $B^\theta_{p,p}([0,1]^n)$ be dense in $L^p([0,1]^n)$, and this requires that
%$\theta<1$. It follows that $\theta_p^*(X)\le 1$. On the other hand, when $\theta<1$ the results of~\cite{BBS} tells us that
%$B^\theta_{p,p}(X)$ is dense in $L^p(X)$ as the class of Lipschitz continuous functions forms a dense subclass of  both spaces.
%Thus we have that $\theta_p^*(X)=1=\pwalk/p$.
%
%	Now we consider the case that $X$ is the glued Sierpi\'nski carpet.  
		By \cite[Theorems~1.1 and~1.4]{MS}, $B^{\pwalk/p}_{p,\infty}(K^{\pm}) \cap C(K^{\pm})$ is dense in $C(K^{\pm})$ for any $p \in (1,\infty)$.
		Hence we can show $\theta_{p}(X) = \pwalk/p$ when $\pwalk > \hdim$ in the same way as Theorem \ref{thm:critical.gluedSG}.
		Assume that $\pwalk \le \hdim$.
		Since $\chi_{K^{+}} \in B^{\theta}_{p,p}(X)$ if and only if $\theta < \hdim/p$, we have $\theta_{p}(X) \ge \hdim/p$.
		To see that $\theta_{p}(X) \le \hdim/p$, let $\theta > \hdim/p \ge \pwalk/p$ and let $u \in B^{\theta}_{p,p}(X)$.
		Then by Lemma~\ref{lem:Besov-KS} we know that $u\in KS^\theta_p(X)$
		and so by Lemma~\ref{lem:B=KS}\ref{it:besovrelation} we also have that $u\in B^{\pwalk/p}_{p,p}(X)$.
		Note that then $u\vert_{K^{\pm}}\in B^{\pwalk/p}_{p,p}(K^{\pm})$.
		Now by Lemma~\ref{lem:Besov-KS} again, we know that 
		$\norm{u|_{K^{+}}}_{KS^{\pwalk/p}_{p}(K^{+})} = \norm{u|_{K^{-}}}_{KS^{\pwalk/p}_{p}(K^{-})} = 0$.
		Hence we have
		from \cite[Theorems~1.1 and 1.4]{MS} that $u|_{K^{+}}$ and $u|_{K^{-}}$ are constant. Since
		$\chi_{K^{+}} \not\in B^{\theta}_{p,p}(X)$, $u$ has to be a constant function, whence it follows that $\theta_{p}(X) \le \hdim/p$.

		Next we prove that $\theta_{p}^{\ast}(X) = \pwalk/p$.
		Since $B^{\pwalk/p}_{p,\infty}(K^{\pm}) \cap C(K^{\pm})$ is dense in $C(K^{\pm})$,
		we can show that
		$\theta_{p}^{\ast}(X) \ge \pwalk/p$ in the same manner as in the proof of Theorem~\ref{thm:critical.gluedSG}.
		Since $B^{\theta}_{p,\infty}(K^{+})$ and $B^{\theta}_{p,\infty}(K^{-})$ have only constant functions when $\theta > \pwalk/p$, $B^{\theta}_{p,\infty}(X)$ can not be dense in $L^{p}(X,\mu)$ for such $\theta$.
		Hence, by Lemma~\ref{lem:B=KS}~\ref{it:besovrelation},
		$B^{\theta}_{p,p}(X)$ is not dense in $L^{p}(X,\mu)$ for any $\theta > \pwalk/p$, 
		from which it follows that $\theta_{p}^{\ast}(X) \le \pwalk/p$.
	\end{proof}

	The following proposition is an analog of Proposition~\ref{prop:KSdiff.bowtie} where now $X$ is the
	glued Sierpi\'nski carpet. In this case, when $p$ is the Ahlfors regular conformal dimension $d_{\mathrm{ARC}}$ of the carpet,
	we must have $\theta_{p}(X) = \theta_{p}^{\ast}(X)$. 

    \begin{prop}\label{prop:KSdiff.gluedSC}
    	Let $X$ be the glued Sierpi\'nski carpet and let $p = d_{\mathrm{ARC}}$. 
    	Set $E_{1} \coloneqq K^{+}$ and $E_{2} \coloneqq K^{-}$ for ease of notation.
    	\begin{enumerate}[\rm(1)]
		\item\label{it:determine.gluedSC} It follows that
		\begin{equation*}
			KS^{\theta_p}_{p}(X)
			= \biggl\{ u_{1}\chi_{E_{1}} + u_{2}\chi_{E_{2}} \biggm|
			\begin{minipage}{160pt}
				$u_{i} \in L^{p}(X,\mu), u_{i}|_{E_{i}} \in KS^{\theta_p}_{p}(E_{i})$,
				$i \in \{ 1,2 \}$,
				$I_{KS}(u_{1},u_{2}) < \infty$
			\end{minipage}
			\biggr\},
		\end{equation*}
		where
		\begin{equation*}
			I_{KS}(u_{1},u_{2}) \coloneqq \limsup_{r \to 0^+}\int_{E_{1} \cap B(o,r)}\int_{E_{2} \cap B(o,r)}\frac{\abs{u_{1}(x) - u_{2}(y)}^{p}}{r^{\hdim + p\theta_{p}}}\,dy\,dx.
		\end{equation*}
		\item\label{it:relation.gluedSC} $KS^{\theta_{p}}_{p}(X) \subsetneq \{ u_{1}\chi_{E_{1}} + u_{2}\chi_{E_{2}} \mid u_{i} \in L^{p}(X,\mu), u_{i}|_{E_{i}} \in KS^{\theta_{p}}_{p}(E_{i}), i \in \{ 1,2 \} \}$.
	\end{enumerate}
	\end{prop}

	\begin{proof}
		The proof of~\ref{it:determine.gluedSC} can be obtained via minor modifications of the proof of
		Proposition~\ref{prop:KSdiff.bowtie}~\ref{it:determine.bowtie}, and we leave it to the interested reader to verify.
		By \cite[Proof of Theorem~2.7]{CCK}, \cite[Theorem~1.4]{MS} and the fact that $\pwalk = \hdim$ when $p = d_{\mathrm{ARC}}$ (see \cite[Remark~9.17]{MS}), there exists $v \in KS^{\theta_{p}}_{p}(K^{+})$ such that
		$\lim_{r \to 0^+}\essinf_{K^+ \cap B(o,r)}\abs{v} = \infty$.
		Once we obtain such a discontinuous function, then
		using the zero-extension $u$ of such a function $v$ to $K^-$, the proof of
		Proposition~\ref{prop:KSdiff.bowtie} verbatim tells us that $u \not\in KS^{\pwalk/p}_{p}(X)$. 
		The proof of \ref{it:relation.gluedSC} is now completed.
	\end{proof}
\end{example}

\section{Proof of Theorem~\ref{thm:main1}}

We now prove Theorem~\ref{thm:main1}; the proof is broken down step by step by the following lemmata. 

\begin{lem}\label{lem:bounded}
Let $\mu$ be a doubling measure on $X$. Suppose that $B^\theta_{p,p}(X)$ is $k$-dimensional for some $k\in {\mathbb N}$ as a vector space
{\rm (}hence $B^\theta_{p,p}(X)\ne \{0\}${\rm )}.  
Then the following hold. 
\begin{enumerate} [label=\textup{(\roman*)},align=left,leftmargin=*,topsep=2pt,parsep=0pt,itemsep=2pt]
	\item\label{it:besovbdd} Every function in $B^\theta_{p,p}(X)$ is bounded.
	\item\label{it:besovsimple} Every function $f\in B^\theta_{p,p}(X)$ is a simple function. Moreover, if 
	$\mu(X)<\infty$ and $k=1$, then
	$f$ is necessarily constant, and if 
	$\mu(X)<\infty$ and $k>1$ or $\mu(X)=\infty$ and 
	$k\ge 1$, then outside of a set of measure zero, $f$ takes on at most $k+1$ values.
	\item\label{it:besovdecomp} Suppose $k>1$.
		Then there is a collection of measurable subsets $E_i$, $i=1,\cdots, k$, of $X$ such that the collection
		$\{\chi_{E_i}\, :\, 1\le i\le k\}$ forms a basis for $B^\theta_{p,p}(X)$ and in addition, $0<\mu(E_i)<\infty$ for each
		$i=1,\cdots, k$, $\mu(E_i\cap E_j)=0$ whenever $i\ne j$, and 
		 if in addition we have that $\mu(X)<\infty$, then $\mu(X\setminus\bigcup_{j=1}^kE_j)=0$. 
\item $B^\theta_{p,p}(X)=\oplus_{i=1}^k B^\theta_{p,p}(E_i)$ as sets.
Moreover, the dimension of $B^\theta_{p,p}(E_i)$ is $1$ for all
$i=1,\cdots, k$. 
\end{enumerate}
\end{lem}

\begin{proof} {\bf Proof of (i):} Suppose that the dimension of $B^\theta_{p,p}(X)$ is finite and that there is an unbounded function
$f\in B^\theta_{p,p}(X)$. By considering $f_+, f_-$ separately, we may consider without loss of
generality that $f\ge 0$ (note that if $f\in  B^\theta_{p,p}(X)$, then $f_+, f_-\in B^\theta_{p,p}(X)$ by Lemma \ref{lem:normal_cont}).
Then we can find a strictly increasing sequence of positive integers $(n_i)_{i\in\N}$ such that
$\mu(f^{-1}((n_i,n_{i+1}]))>0$ for each $i\in\N$. Set
\begin{equation*}\label{e:cut.levelset}
	f_i(x):=\max\{f(x)-n_i,\, 0\},
\end{equation*}
then $f_i\in B^\theta_{p,p}(X)$ by Lemma \ref{lem:normal_cont}.

Note that $f_1$ is not a linear combination of any of up to $\ell$ many choices of functions
$f_{i_1},\cdots, f_{i_\ell}$ with $i_1,\cdots, i_\ell$ distinct from $1$, for all such linear combinations will vanish on the set
$f^{-1}((n_1,n_2])$ where $f_1$ is nonzero. Note also that $f_2$ cannot be a linear combination of $f_1$ and other $f_j$, $j\ne 2$,
either, as on the set $f^{-1}((n_2,n_3])$ the functions $f_j$, $j\ge 3$, vanish and so if $f_2$ were to be such a linear combination,
on that set we must have $f_2=af_1$ for some $a\ne 0$. This also is not possible as $f_1$ is nonzero on the set $f^{-1}((n_1,n_2])$
and $f_2$ and all $f_j$, $j>2$, vanish there. Hence $f_1$ and $f_2$ are linearly independent of each other and of all the other $f_j$, $j\ge 3$.
We have also proved that $\sum_{j=1}^2a_jf_j=0$ on $f^{-1}((n_1,n_3])$ implies that $a_1=a_2=0$.

Now we proceed by induction. Suppose we have shown that $f_1,\cdots, f_i$ are linearly independent of each other and of all the
other $f_j$, $j\ge i+1$ and that $\sum_{j=1}^ia_jf_j=0$ on $f^{-1}((n_1,n_{i+1}])$ implies that $a_j=0$ for $j=1,\cdots, i$.
We wish to show that $f_{i+1}$ is also independent of the other functions $f_j$, $j\ne i+1$. Indeed, if it is not,
then by considering the set $f^{-1}((n_1,n_{i+2}])$, we see that on this set we must have
$f_{i+1}=\sum_{j=1}^ia_i\, f_i$ with at least one of
$a_i$ nonzero. But then, on the set $f^{-1}((n_1,n_{i+1}])$ we have that $\sum_{j=1}^ia_j\, f_j=0$,
which then indicates that each $a_j=0$ for
$j=1,\cdots, i$. That is, $f_{i+1}$ cannot be a linear combination of the other functions $f_j$, $j\ne i$. It follows that
the collection $\{f_i\, :\, i\in\N\}$ is a linearly independent subcollection of $B^\theta_{p,p}(X)$, violating the finite
dimensionality of $B^\theta_{p,p}(X)$. Thus $f$ must be bounded. \\
{\bf Proof of (ii):} 
Let $f\in B^\theta_{p,p}(X)$ such that $f$ is not the zero function. 
Then both $f_+$ and $f_-$ are in $B^\theta_{p,p}(X)$, and so we first focus on
the possibility that $f\ge 0$ with $f \not\equiv 0$. 
We want to prove that there are positive real numbers $b_1,b_2,\cdots, b_l$ with $l\le k$ and $b_{i} < b_{i+1}$ for $i = 1,\dots,l-1$ such that
\[
\mu(X\setminus f^{-1}(\{b_1,\cdots, b_l,0\}))=0,
\]
or equivalently the support of the push-forward measure $f_{\ast}\mu$, $\mathrm{supp}(f_{\ast}\mu)$, is contained in $\{b_1,\cdots, b_l,0\}$. (Note that $\mathrm{supp}(f_{\ast}\mu) \subseteq [0,\norm{f}_{L^{\infty}}]$ since $f$ is bounded 
due to (i).)   
We prove this by contradiction.
Suppose the above claim fails. 
Then recalling that $f\ge 0$,
\[
%\#\bigl(\mathrm{supp}(f_{\ast}\mu)\bigr) =
\#\bigl\{ \alpha \in \mathbb{R} \bigm| \alpha\ge 0 \text{ and $\mu(\{ \abs{f-\alpha} < \varepsilon \}) > 0$ for all $\varepsilon > 0$} \bigr\} \ge k+2
\]
and we can find positive 
numbers
$a_2,\cdots, a_{k+2}$ with $a_{1} = 0$, $a_{k+2} = \norm{f}_{L^{\infty}}$ and $a_i<a_{i+1}$ for $i=1,\cdots, k+1$, such that
$\mu(f^{-1}((a_i,a_{i+1}]))>0$ for $i=1,\cdots, k+1$. 

As in the proof of (i), we consider the
functions $f_i$, $i=1,\cdots, k+1$, given by
\[
f_i(x)=\max\{f(x)-a_i,\, 0\}.
\]
Since $a_i\ge 0$, it follows from Lemma \ref{lem:normal_cont} that $f_i\in B^\theta_{p,p}(X)$. 
%that $0\le f_i\le f$, and hence $f_i\in L^p(X)$, and so $f_i\in B^\theta_{p,p}(X)$. 
Now a repeat of the proof of (i) tells us that the collection
$\{f_1,\cdots, f_{k+1}\}\subset B^\theta_{p,p}(X)$ is linearly independent, violating the hypothesis that
the dimension of $B^\theta_{p,p}(X)$ is $k$. The claim now
follows for non-negative functions that are not identically zero.
In particular, for such functions, we can set $E_i:=f^{-1}(\{b_i\})$ for $i=1,\cdots, l\le k$, and 
see that 
\[
f=\sum_{i=1}^l b_i\, \chi_{E_i}.
\]
We now set $b_0:=0$, and by Lemma~\ref{lem:normal_cont},
note that for $i=1,\cdots, l$, the function $h_i$ given by $h_i(x)=\max\{0,\min\{f(x)-b_{i-1}, b_i-b_{i-1}\}\}$
belongs to $B^\theta_{p,p}(X)$ 
with $h_i=(b_i-b_{i-1})\chi_{F_{i}}$, where $F_{i} \coloneqq \bigcup_{j=i}^{l}E_{j}$. It follows that $\chi_{F_{i}}=(b_i-b_{i-1})^{-1}\, h_i\in B^\theta_{p,p}(X)$ and hence $\chi_{F_{i}} \in B^{\theta}_{p,p}(X)$. It follows that $\chi_{E_i}\in B^\theta_{p,p}(X)$ as well
for $i=1,\cdots, l$. Note that $\mu(E_{i} \cap E_{j}) = 0$ when $i \neq j$.

If $f$ is not non-negative and not identically zero, then we apply the above conclusion to $f_{+}$ and $f_{-}$ separately, and
so we have distinct positive numbers $a_1,\cdots, a_j$ and distinct positive numbers $b_1,\cdots, b_l$ with 
$j,l\le k$, 
and measurable sets $E_1,\cdots, E_j$ and $F_1,\cdots, F_l$ such that 
\[
f=f_+-f_-=\sum_{i=1}^ja_i\, \chi_{E_i}-\sum_{m=1}^l b_m\, \chi_{F_m}.
\]
We can also ensure that $\mu(E_i\cap F_m)=0$ for all $(i,m)$. Moreover, as $f\in L^p(X)$, we must have
$\mu(E_i)$ and $\mu(F_m)$ are finite whenever $1\le i\le j$ and $1\le m\le l$. Thus the collection
$\{\chi_{E_i},\, \chi_{F_m}\, :\, i\in\{1,\cdots, j\}, m\in\{1,\cdots, l\}\}$ is a linearly independent collection of 
functions in $B^\theta_{p,p}(X)$, and hence we must have that $m+l\le k$, that is,
there are at most $k$ non-zero real numbers $c_1,\cdots, c_n$ such that
\[
\mu(X\setminus f^{-1}(\{c_1,\cdots, c_n,0\}))=0.
\]
\\
{\bf Proof of (iii):} Let $\{f_1,\cdots, f_k\}$ be a basis for $B^\theta_{p,p}(X)$. By (ii), we know that for
each $j=1,\cdots, k$ there are measurable subsets $E_{j,1},\cdots, E_{j,N_j}$ of $X$ 
with $\chi_{E_{j,i}} \in B^\theta_{p,p}(X)$ and \emph{distinct} non-zero
real numbers $a_{j,1},\cdots, a_{j,N_j}$ such that
\[
f_j=\sum_{i=1}^{N_j}a_{j,i}\, \chi_{E_{j,i}}.
\]
We can make this simple-function decomposition of $f_j$ so that $\mu(E_{j,i}\cap E_{j,k})=0$ for
$i,k\in\{1,\cdots, N_j\}$ with $i\ne k$ and in addition we require that $\mu(E_{j,i})>0$ for each
$i=1,\cdots, N_j$. 

Next, we break the sets $E_{j,i}$, $j=1,\cdots, k$ and $i=1,\cdots, N_j$ into pairwise disjoint subsets as
follows. Observing that $\mu(E_{j,i}\cap E_{j,n})=0$ if $i\ne n$, it suffices to consider pairs of sets
$E_{j,i}$ and $E_{m,n}$ with $j\ne m$.
Since $\chi_{E_{j,i}}$ and $\chi_{E_{m,n}}$ are in $B^\theta_{p,p}(X)$, it follows from Lemma~\ref{lem:leibniz} that 
the function $\chi_{E_{j,i}\cap E_{m,n}}=\chi_{E_{j,i}}\, \chi_{E{m,n}}$ is also in $B^\theta_{p,p}(X)$.
If $\mu(E_{j,i}\cap E_{m,n})>0$ and $\mu(E_{j,i}\Delta E_{m,n})>0$,
then we can replace $E_{j,i}$ and $E_{m,n}$ with $E_{j,i}\cap E_{m,n}$, and $E_{j,i}\setminus E_{m,n}$
if $\mu(E_{j,i}\setminus E_{m,n})>0$ and $E_{m,n}\setminus E_{j,i}$ if $\mu(E_{m,n}\setminus E_{j,i})>0$
(note that in the case considered here, we must have at least one of $\mu(E_{m,n}\setminus E_{j,i})$
and $\mu(E_{j,i}\setminus E_{m,n})$ is positive).

Since the collection $\{E_{j,i}\, :\, j=1,\cdots, k, i=1,\cdots, N_j\}$ is a finite collection of sets, the above procedure
involving each pair of sets from this collection needs to be done only finitely many times; thus we obtain the
collection of sets $E_i$, $i=1,\cdots, N$ such that
\begin{equation}\label{eq:disjoint}
\mu(E_i\cap E_j)=0 \text{ whenever } i\ne j.
\end{equation}

As each $f_j$ is a linear combination of the characteristic functions of $E_{j,i}$, $i=1,\cdots, N_j$, it follows
that $f_j$ is a linear combination of the characteristic functions $\chi_{E_i}$, $i=1,\cdots, N$. Because
the collection $\{f_1,\cdots, f_k\}$ spans $B^\theta_{p,p}(X)$, the collection
$\{\chi_{E_i}\, :\, i=1,\cdots, N\}$ spans $B^\theta_{p,p}(X)$ as well. Moreover, by~\eqref{eq:disjoint}
this collection of functions is also linearly independent; hence $N=k$, and this collection forms a basis
for $B^\theta_{p,p}(X)$.

Finally, note that when $\mu(X)<\infty$, 
the constant function $u\equiv 1$ is in $B^\theta_{p,p}(X)$, and so necessarily 
$u=\sum_{j=1}^k\chi_{E_j}$, that is, $\mu(X\setminus\bigcup_{j=1}^kE_j)=0$. \\
{\bf Proof of (iv):} By (iii), it is enough to show that $B^\theta_{p,p}(E_i)$ consists only of constant
functions (i.e. the dimension of $B^\theta_{p,p}(E_i)$ is $1$) for
all $i=1,\cdots, k$.
Now suppose there is $i\in \{1,\cdots, k\}$ and a
non-constant $g\in B^\theta_{p,p}(E_i)$.
By Lemma~\ref{lem:normal_cont}, we may assume that $g$ is bounded.
Since $\chi_{E_i}\in B^\theta_{p,p}(X)$, we have
\begin{align}
||\chi_{E_i}||_{B^\theta_{p,p}(X)}^p
&=\int_{E_i^c}\int_{E_i}\frac{1}{d(x,y)^{\theta p}\, \mu(B(x,d(x,y)))}\, d\mu(y)\, d\mu(x)\nonumber\\
+& \int_{E_i}\int_{E_i^c}\frac{1}{d(x,y)^{\theta p}\, \mu(B(x,d(x,y)))}\, d\mu(y)\, d\mu(x)<\infty.~~~~~\label{eq:id-int}
\end{align}
Now define $\widetilde g: X\to \mathbb R$ by $\widetilde g=g_i\chi_{E_i}$, that is,
$\widetilde g|_{E_i}=g$ and $\widetilde g|_{E_i^c}=0$. Then
$\|\widetilde g\|_{L^p(X)}^p=\|g\|_{L^p(E_i)}^p<\infty$ and
\begin{eqnarray*}
||\widetilde g||_{B^\theta_{p,p}(X)}^p& \le &
||g||_{B^\theta_{p,p}(E_i)}^p
+\int_{E_i^c}\int_{E_i}\frac{|g(y)|^p}{d(x,y)^{\theta p}\, \mu(B(x,d(x,y)))}\, d\mu(y)\, d\mu(x)\\
&&~\qquad\qquad~~+ \int_{E_i}\int_{E_i^c}\frac{|g(x)|^p}{d(x,y)^{\theta p}\, \mu(B(x,d(x,y)))}\, d\mu(y)\, d\mu(x)\\
&\le & ||g||_{B^\theta_{p,p}(E_i)}^p+\|g\|_{L^\infty(X)}^p
||\chi_{E_i}||_{B^\theta_{p,p}(X)}^p<\infty,
\end{eqnarray*}
where the last inequality is due to~\eqref{eq:id-int}.
It follows that $\widetilde g\in B^\theta_{p,p}(X)$, and so by~(iii) there are real numbers $a_1,\cdots, a_k$
such that $\widetilde{g}=\sum_{j=1}^ka_j\chi_{E_j}$, which in turn means that $\widetilde g$ (and hence $g$) is constant
$\mu$-a.e.~in $E_i$, contradicting the non-constant nature of $g$. It follows that every function in $B^\theta_{p,p}(E_i)$ must
be constant. 
\end{proof} 

\begin{remark}
Lemma \ref{lem:bounded} proves claims~\ref{it:positivemass},~\ref{it:cover},~\ref{it:basis}
and~\ref{it:directsum} of Theorem~\ref{thm:main1}. 
Lemma~\ref{lem:Besov-KS}
verifies claim~\ref{it:zeroenergy} of Theorem~\ref{thm:main1}. Claim~\ref{it:thetap} of Theorem~\ref{thm:main1} follows consequently from the definition of $\theta_p(X)$.
\end{remark}

\begin{lem}\label{lem:invariantEi}
Under the hypotheses of Lemma~\ref{lem:bounded} above, and with the sets $E_i$, $i=1,\cdots, k$, as
constructed in that lemma, we have that $u\chi_{E_i}\in KS^\theta_p(X)$
whenever $u\in KS^\theta_p(X)$ is bounded. 
\end{lem}

\begin{proof}
The claim follows immediately from combining Lemma~\ref{lem:leibniz} and
the fact that $\chi_{E_{i}} \in B^{\theta}_{p,p}(X)$. 
\end{proof} 

Finally, the next lemma verifies~\ref{it:energydecomp} of Theorem~\ref{thm:main1} and completes the proof of Theorem~\ref{thm:main1}.

\begin{lem}\label{lem:main.sub6}
Under the setting of Theorem~\ref{thm:main1}, claim~\ref{it:energydecomp} holds true.
\end{lem}

\begin{proof}
Let $u\in KS^\theta_p(X)$ such that $\|u\|_{L^\infty(X)}=:M$ is bounded. Then
\begin{align*}
\int_X\vint{B(x,r)}&\frac{|u(x)\chi_{E_j}(x)-u(y)\chi_{E_j}(y)|^p}{r^{\theta p}}\, d\mu(y)\, d\mu(x)\\
&=\int_{E_j}\int_{B(x,r)\cap E_j}\frac{|u(y)-u(x)|^p}{r^{\theta p}\, \mu(B(x,r))}\, d\mu(y)\, d\mu(x)\\
    &\qquad+\int_{E_j}\int_{B(x,r)\setminus E_j}\frac{|u(x)\chi_{E_j}(x)|^p}{r^{\theta p}\, \mu(B(x,r))}\, d\mu(y)\, d\mu(x)\\
    &\qquad+\int_{X\setminus E_j}\int_{B(x,r)\cap E_j}\frac{|u(y)\chi_{E_j}(y)|^p}{r^{\theta p}\, \mu(B(x,r))}\, d\mu(y)\, d\mu(x).
\end{align*}
Note that
\begin{align*}
\int_{E_j}&\int_{B(x,r)\setminus E_j}\frac{|u(x)\chi_{E_j}(x)|^p}{r^{\theta p}\, \mu(B(x,r))}\, d\mu(y)\, d\mu(x)\\
   &\qquad\qquad\qquad+ \int_{X\setminus E_j}\int_{B(x,r)\cap E_j}\frac{|u(y)\chi_{E_j}(y)|^p}{r^{\theta p}\, \mu(B(x,r))}\, d\mu(y)\, d\mu(x)\\
  \le &M^p\int_{E_j}\int_{B(x,r)\setminus E_j}\frac{|\chi_{E_j}(x)|^p}{r^{\theta p}\, \mu(B(x,r))}\, d\mu(y)\, d\mu(x)\\
  &\qquad\qquad\qquad+ M^p\int_{X\setminus E_j}\int_{B(x,r)\cap E_j}\frac{|\chi_{E_j}(y)|^p}{r^{\theta p}\, \mu(B(x,r))}\, d\mu(y)\, d\mu(x)\\
  =&M^p\int_{E_j}\int_{B(x,r)\setminus E_j}\frac{|\chi_{E_j}(x)-\chi_{E_j}(y)|^p}{r^{\theta p}\, \mu(B(x,r))}\, d\mu(y)\, d\mu(x)\\
  &\qquad\qquad\qquad+ M^p\int_{X\setminus E_j}\int_{B(x,r)\cap E_j}\frac{|\chi_{E_j}(x)-\chi_{E_j}(y)|^p}{r^{\theta p}\, \mu(B(x,r))}\, d\mu(y)\, d\mu(x)\\
  \le &M^{p}\int_X\vint{B(x,r)}\frac{|\chi_{E_j}(x)-\chi_{E_j}(y)|^p}{r^{\theta p}}\, d\mu(y)\, d\mu(x),
\end{align*}
and thanks to~\ref{it:zeroenergy} of Theorem~\ref{thm:main1} (verified above),
the last expression above tends to $0$ as $r\to 0^+$. It follows that
\begin{align*}
\|u\chi_{E_j}\|_{KS^\theta_p(X)}^p&=\limsup_{r\to 0^+}\int_X\vint{B(x,r)}\frac{|u(x)\chi_{E_j}(x)-u(y)\chi_{E_j}(y)|^p}{r^{\theta p}}\, d\mu(y)\, d\mu(x)\\
&=\limsup_{r\to 0^+}\int_{E_j}\int_{B(x,r)\cap E_j}\frac{|u(y)-u(x)|^p}{r^{\theta p}\, \mu(B(x,r))}\, d\mu(y)\, d\mu(x),
\end{align*}
completing the proof.
\end{proof}

\section{Proof of Theorem~\ref{thm:main2} and Theorem~\ref{thm:main3.general}}

In this section we provide a proof of the remaining two main results of this paper.

\begin{proof}[Proof of Theorem~\ref{thm:main2}] 
It suffices to show that any function in $B^{\theta}_{p,p}(X)$ is a constant function, in particular, the dimension of 
$B^{\theta}_{p,p}(X)$ is $1$ if $\mu(X) < \infty$, and $B^{\theta}_{p,p}(X) = \{ 0 \}$ if $\mu(X) = \infty$.
Suppose there is a non-constant function $g\in B^\theta_{p,p}(X)$.
Since $g$ is non-constant, at least one of $g_+$ and $g_-$ is non-constant; hence, without loss of generality, we may assume that
$g\ge 0$ on $X$. Then there is a positive real number $a$ such that $\mu(g^{-1}([a,\infty)))>0$ and $\mu(g^{-1}([0,a)))>0$.
We can then find a positive real number $\delta<a$ such that $\mu(g^{-1}([0,a-\delta]))>0$ as well. Now by 
Lemma~\ref{lem:normal_cont}
and Lemma~\ref{lem:Besov-KS}, we know that $g_{a,\delta}:=\max\{0,\min \{g-(a-\delta), \delta\}\}\in B^\theta_{p,p}(X)\subset KS^\theta_p(X)$
with $\|g_{a,\delta}\|_{KS^\theta_p(X)}=0$. On the other hand, the choices of $a$ and $\delta$ mean that
$\|g_{a,\delta}\|_{B^\theta_{p,\infty}(X)}>0$, violating condition~\ref{e:wmax}. Thus no such $g$ exists. 
\end{proof}

\begin{proof}[Proof of Theorem~\ref{thm:main3.general}]
	In~\cite[Theorem~1.5]{GYZ}, a condition called property~(NE) is assumed in addition; however, the
		proof of inequality~(2.8) in the proof of that theorem in~\cite{GYZ} does not need this property, and so we can
		use~\cite[(2.8)]{GYZ} verbatim in our setting. Now, by~\cite[(2.8)]{GYZ} 
	and by~\cite[Theorem 5.2]{GKS}, there exists $C \ge 1$ such that for any $u \in B^{\theta}_{p,\infty}(X)$,
	\begin{equation*}
		\liminf_{t \to 0^+}\int_{X}\vint{B(x,t)}\frac{\abs{u(x) - u(y)}^{p}}{t^{p\theta}}\,d\mu(y)\,d\mu(x) \\
		\le C\liminf_{\theta'\to \theta^-}(\theta-\theta')\norm{u}_{B^{\theta'}_{p,p}(X)}^{p}.
	\end{equation*}
	Now suppose that there is a non-constant function $u\in B^{\theta}_{p,p}(X)$.
	Then we have by the Lebesgue dominated convergence theorem that
	\[
	\lim_{\theta'\to \theta^{-}}\norm{u}_{B^{\theta'}_{p,p}(X)}^{p}
	= \norm{u}_{B^{\theta}_{p,p}(X)}^{p} > 0,
	\]
	but then
	\[
	\liminf_{\theta'\to \theta^-}(\theta-\theta')\norm{u}_{B^{\theta'}_{p,p}(X)}^{p} = 0,
	\]
	whence it follows from \eqref{e:KSSob} that $\int_{X}\abs{u - u_{X}}^{p}\,d\mu = 0$.
	Hence $u$ must be constant on $X$, which is a contradiction of the supposition that $u$ is non-constant on $X$.
	Therefore $B^{\theta}_{p,p}(X)$ consists only of constant functions.
\end{proof}

\begin{proof}[Proof of Corollary~\ref{cor:main3}]
	Under the hypotheses of Corollary~\ref{cor:main3}, we obtain $\theta_{p}(X) = 1$ and \eqref{e:KSSob} by \cite[Theorem 5.1]{Bau} and \cite[Theorem 10.5.2]{HKSTbook}, so we can apply Theorem~\ref{thm:main3.general}.
\end{proof} 

\noindent {\bf Acknowledgements.}\,\,  
The authors are grateful to the anonymous referee for helpful comments that improved the exposition of this paper.

\end{document}